\documentclass[12pt,linesnumbered,ruled,vlined]{article}
\usepackage{mathptmx}
\usepackage[utf8x]{inputenc}
\usepackage{color}
\usepackage{array}
\usepackage{booktabs}
\usepackage{mathtools}
\usepackage{amsmath}
\usepackage{amsthm}
\usepackage{amssymb}
\usepackage{graphicx}
\usepackage{multirow} 
\usepackage[unicode=true,pdfusetitle,
 bookmarks=true,bookmarksnumbered=false,bookmarksopen=false,
 breaklinks=false,pdfborder={0 0 1},backref=false,colorlinks=true]
 {hyperref}
\hypersetup{
 pdfborderstyle={},pdfborderstyle={},pdfborderstyle={},urlcolor=blue,linkcolor=blue,citecolor=blue}

\makeatletter

\@ifundefined{date}{}{\date{}}

\usepackage{amsfonts}
\usepackage{dsfont}

\numberwithin{equation}{section}

\allowdisplaybreaks[4]
\DeclareMathAlphabet{\mathcal}{OMS}{cmsy}{m}{n}

\newtheorem{thm}{Theorem}[section]
\newtheorem{lem}[thm]{Lemma}

\newtheorem{prop}[thm]{Proposition}
\newtheorem{cor}[thm]{Corollary}

\newtheorem{rem}[thm]{Remark}

\@addtoreset{equation}{section}

\makeatother

\usepackage{authblk}
\usepackage{tikz}

\begin{document}

\title{L\'evy Area Analysis and Parameter Estimation for fOU Processes via Non-Geometric Rough Path Theory\footnote{Published in the journal: \textit{Acta Mathematica Scientia, 2024}.}}


\author[1]{Zhongmin Qian\thanks{qianz@maths.ox.ac.uk}}
\author[2,3]{Xingcheng Xu\thanks{xingcheng.xu18@gmail.com}}
\affil[1]{Mathematical Institute, University of Oxford}
\affil[2]{Shanghai Artificial Intelligence Laboratory}
\affil[3]{School of Mathematical Sciences, Peking University}

\date{}

\maketitle

\begin{abstract}
This paper addresses the estimation problem of an unknown drift parameter matrix for a fractional Ornstein-Uhlenbeck process in a multi-dimensional setting. To tackle this problem, we propose a novel approach based on rough path theory that allows us to construct pathwise rough path estimators from both continuous and discrete observations of a single path. Our approach is particularly suitable for high-frequency data. To formulate the parameter estimators, we introduce a theory of pathwise It\^o integrals with respect to fractional Brownian motion. By establishing the regularity of fractional Ornstein-Uhlenbeck processes and analyzing the long-term behavior of the associated L\'evy area processes, we demonstrate that our estimators are strongly consistent and pathwise stable. Our findings offer a new perspective on estimating the drift parameter matrix for fractional Ornstein-Uhlenbeck processes in multi-dimensional settings, and may have practical implications for fields including finance, economics, and engineering.
\end{abstract}
\vspace{2mm}
\hspace{11mm}{\bf Keywords}. It\^o integration, L\'evy area, Non-geometric rough path, FOU processes, 

\hspace{5mm}Pathwise stability, Long time asymptotic, High-frequency data.
\vspace{2mm}

\hspace{5mm}{\bf MSC{[}2010{]}}: 60H05, 62F12, 62M09, 91G30.

\section{Introduction}

\label{intro-sec}

The field of statistical analysis of time series and random processes involves parameter and non-parameter estimations, and  statistical inferences as well. The majority of research in this area has focused on models described in terms of diffusion processes and semi-martingales. Standard references, such as \cite{JS03,LS77,Rao99-a,Rao99}, are among those that have concentrated on these models. However, applications that require the consideration of long-time memory effects have brought some attention to models that are not semi-martingales, such as those discussed in \cite{HN10,KLB02,TV07}.

This article focuses on multi-dimensional Ornstein-Uhlenbeck (OU) processes driven by fractional Brownian motions (fBM). These processes are commonly referred to as fractional Ornstein-Uhlenbeck (fOU) processes and are defined as the solution to the stochastic differential equation (SDE)
\begin{equation}
dX_{t}=-\Gamma X_{t}dt+\Sigma dB_{t}^{H},\ X_{0}=x_{0}.\label{fOU-h}
\end{equation}
Here, $B^{H}$ is a $d$-dimensional fBM with a Hurst parameter $H\in(0,1)$, $\Gamma\in\mathbb{R}^{d\times d}$ is the drift matrix which is symmetric and positive-definite, and $\Sigma\in\mathbb{R}^{d\times d}$ is the non-degenerate volatility matrix. The SDE must be interpreted as the stochastic integral equation
\[
X_{t}=x_{0}-\int_{0}^{t}\Gamma X_{s}ds+\Sigma B_{t}^{H},
\]
which has a unique solution given by

\begin{equation}
X_{t}=e^{-\Gamma t}x_{0}+\int_{0}^{t}e^{-\Gamma(t-s)}\Sigma dB_{s}^{H}.
\end{equation}
The integral on the right-hand side is a Young's integral. Therefore, like ordinary OU processes, $(X_{t})$ is a Gaussian process.

The multi-dimensional fOU processes can be used to describe systems with linear interactions perturbed by Gaussian noise. Inter-bank lending is a real-world example that can be modeled using Equation (\ref{fOU-h}), as demonstrated in \cite{CFS15,FI13}. A crucial question in such applications is to estimate the interaction structure $\Gamma$ from a single path observation of the process, assuming that $\Sigma$ is known and the single path $X(\omega)$ can be continuously or discretely observed.

In the one-dimensional case, maximum likelihood estimators (MLE) and least square estimators (LSE) have been studied extensively, and their properties have been documented in literature \cite{HN10,HNZ17,KLB02,TV07}. The MLE based on continuous observation has been studied by Kleptsyna and Le Breton \cite{KLB02} and Tudor and Viens \cite{TV07}, who obtained the strong consistency of the MLE as $T$ goes to infinity. Hu and Nualart \cite{HN10} investigated the LSE for the case where the Hurst parameter is $H>\frac{1}{2}$, and proved its strong consistency as $T\to\infty$ for $H>\frac{1}{2}$. They also established a central limit theorem for $\frac{1}{2}<H<\frac{3}{4}$. The results were extended by Hu, Nualart and Zhou \cite{HNZ17} for all $H\in(0,1)$.

However, few research has been conducted on parameter estimation for multi-dimensional fOU processes. This paper aims to fill this gap. First, we present an estimator based on the rough path theory for continuous observation of a single path. To formulate the parameter estimator, we define an It\^o type integration theory for multi-dimensional fBM.

Coutin and Qian \cite{CQ02} developed a theory of Stratonovich integration for multi-dimensional fBM with Hurst parameter $H>\frac{1}{4}$ using rough path analysis. However, the problem of constructing a rough path theory for fBM with $H\leq\frac{1}{4}$ remains unsolved. Here, we focus on fBM with $H$ such that $\frac{1}{3}<H\leq\frac{1}{2}$, where both fBM and fOU processes have finite $p$-variation with $2\leq p<3$. We canonically enhance these processes to geometric rough paths, which allows us to define It\^o-type integrals with respect to fBM and fOU processes by correcting their enhanced L\'evy area processes. We then apply this theory to investigate the parameter estimation problem for fOU processes based on continuous observation. To establish the strong consistency of the parameter estimator, we also explore the regularity of fOU processes and the long-term asymptotic behavior of their L\'evy area processes, which we believe are interesting in their own right.

We also address the parameter estimation problem for fOU processes based on discrete observation. In practice, observations are often discrete rather than continuous, even though the sampling frequency can be increased, as in the case of high-frequency financial data. To tackle this statistical inference problem, we recommend \cite{Ait02,AJ09,AJ10,AJ14,AMZ11,CG11,MZ09} and related references. In this paper, we construct a parameter estimator based on high-frequency discrete observation using rough path theory and establish its strong consistency. It is worth noting that Diehl, Friz, and Mai \cite{DFM16} used rough path analysis to study maximum likelihood estimators for diffusion processes and initiated research on estimators for the fractional case, but only for small $\varepsilon$ when $H=\frac{1}{2}-\varepsilon$.

The methodology proposed in this paper offers several advantages over existing methods. First, our estimators are applicable to multi-dimensional fOU processes, which reveal the non-trivial role played by L\'evy area processes, and are fundamentally different from the one-dimensional case. Second, the parameter estimators are pathwise defined and can be computed based on observations of a single path. Third, our parameter estimators exhibit pathwise stability and robustness, in the sense that if two observations are close in the so-called $p$-variation distance (as defined in the main text below), then their corresponding estimators are also close. Fourth, our estimators can be constructed using both continuous and discrete observational data, and are particularly useful for high-frequency financial data.

We note that our approach can be extended to the Ornstein-Uhlenbeck process $X_{t}$ driven by a general Gaussian noise $G_{t}$ satisfying certain technical conditions, where
\begin{equation}
X_{t}=e^{-\Gamma t}x_{0}+\int_{0}^{t}e^{-\Gamma(t-s)}\Sigma dG_{s}.
\end{equation}
The integral on the right-hand side is well-defined as long as $t\rightarrow G_{t}$ is $\alpha$-H\"older continuous for some $\alpha>0$. A work that can be referred to in this direction is Chen and Zhou \cite{CZ21}. These singular OU processes may have practical applications.

The structure of this paper is organized as follows. In Section \ref{sec-RP}, we introduce some preliminary concepts regarding rough path theory and present a framework for pathwise It\^o integrals for both fBM and fOU processes. In Section \ref{sec-longtime}, we explore the regularity of fOU processes and examine the long time behavior of their associated L\'evy area processes. Then, in Section \ref{sec-CRPE}, we construct a continuous rough path estimator and present a complete proof for its almost sure convergence and pathwise stability. In Section \ref{sec-DRPE}, we present the discrete rough path estimator that is based on high-frequency data. 

\section{Rough paths and It\^o integration}

\label{sec-RP}

In this section, we present the notations used in the rough path theory, following established references such as \cite{FH14,FV10,Gub04,LCL,LQ02}. We provide a precise definition of It\^o integrals for both fBM and fOU processes.

\subsection{Preliminary of rough paths}



We define the truncated tensor algebra $T^{(2)}(\mathbb{R}^{d})$ as follows:
$T^{(2)}(\mathbb{R}^{d}):=\oplus_{n=0}^{2}(\mathbb{R}^{d})^{\otimes n}.$
Here we adopt the convention that $(\mathbb{R}^{d})^{\otimes 0}=\mathbb{R}$. In this context, we use $\Delta$ to represent the simplex given by $\{(s,t):0\leq s<t\leq T\}$.
Let $X_t$ be a continuous path with finite $p$-variation, where $2<p<3$, defined on the interval $[0,T]$. We denote $\mathbf{X}_{s,t}=(1,X_{s,t},\mathbb{X}_{s,t})$ as an element of the space $T^{(2)}(\mathbb{R}^{d})$, where $X_{s,t}=X_t-X_s\in\mathbb{R}^{d}$ and $\mathbb{X}_{s,t}\in\mathbb{R}^{d}\otimes\mathbb{R}^{d}$. To further illustrate the concept, if the process $X_t$ is of finite variation, then $\mathbb{X}_{s,t}=\int_{s<t_1<t_2<t}dX_{t_1}\otimes dX_{t_2}$, where the integral represents the tensor product of the two differentials. We refer to $\mathbf{X}_{s,t}$ as a lift of the process $X$ to the space $T^{(2)}(\mathbb{R}^{d})$ if it satisfies both finite $p$-variation and Chen's identity.

The initial motivation behind this concept is to define integrals with respect to $X$ by increasing the information on $X$. We recall Chen's identity, which relates to algebraic information, and the definition of finite $p$-variation, which relates to analysis information.

We recall that $\mathbf{X}_{s,t}=(1,X_{s,t},\mathbb{X}_{s,t})$ satisfies
\textit{Chen's identity} if
\begin{equation}
X_{s,t}=X_{t}-X_{s},\label{Chen-r1}
\end{equation}
\begin{equation}
\mathbb{X}_{s,t}-\mathbb{X}_{s,u}-\mathbb{X}_{u,t}=X_{s,u}\otimes X_{u,t},\label{Chen-r2}
\end{equation}
for all $(s,u),\ (u,t)\in\Delta$.

$\mathbf{X}=(1,X_{s,t},\mathbb{X}_{s,t})$ has \textit{finite $p$-variations}
if
\[
\sup_{\mathcal{P}}\sum_{[s,t]\in\mathcal{P}}|X_{s,t}|^{p}<\infty,\ \sup_{\mathcal{P}}\sum_{[s,t]\in\mathcal{P}}|\mathbb{X}_{s,t}|^{p/2}<\infty,
\]
where $\mathcal{P}$ is a partition of $[0,T]$. This is equivalent
to the fact that there exists a control $\omega(s,t)$ such that
\[
|X_{s,t}|\leq\omega(s,t)^{1/p},\ |\mathbb{X}_{s,t}|\leq\omega(s,t)^{2/p},\ \forall(s,t)\in\Delta.
\]
A control $\omega$ is a non-negative, continuous, super-additive
function on $\Delta$ and satisfies that $\omega(t,t)=0$.

Let $2<p<3$ be a constant. A function $\mathbf{X}=(1,X,\mathbb{X})$
from $\Delta$ to $T^{(2)}(\mathbb{R}^{d})$ is called a $p$-\textit{rough
path} if it has finite $p$-variation and satisfies Chen's identity.
We denote the space of $p$-rough paths as $\Omega_{p}(\mathbb{R}^{d})$.

According to Lyons and Qian \cite{LQ02}, the integration operator is
defined as a linear map from $\Omega_{p}(\mathbb{R}^{d})$ to $\Omega_{p}(\mathbb{R}^{e})$,
i.e. $\int F:\Omega_{p}(\mathbb{R}^{d})\to\Omega_{p}(\mathbb{R}^{e})$,
and denote the integral by $Y=\int F(X)d_{\mathfrak{R}}\mathbf{X}$,
where
\begin{equation}
Y_{u,v}^{1}\equiv\int_{u}^{v}F(X)d_{\mathfrak{R}_{1}}\mathbf{X}:=\lim_{|\mathcal{P}|\to0}\sum_{[s,t]\in\mathcal{P}}F(X_{s})X_{s,t}+DF(X_{s})\mathbb{X}_{s,t},\label{level-1}
\end{equation}
and the second level $Y^{2}$ is defined by
\begin{equation}
Y_{u,v}^{2}\equiv\int_{u}^{v}F(X)d_{\mathfrak{R}_{2}}\mathbf{X}:=\lim_{|\mathcal{P}|\to0}\sum_{[s,t]\in\mathcal{P}}Y_{u,s}^{1}\otimes Y_{s,t}^{1}+F(X_{s})\otimes F(X_{s})\mathbb{X}_{s,t},\label{level-2}
\end{equation}
where the limit takes over all finite partitions  $\mathcal{P}$ of interval $[u,v]$.

\subsection{FBM as rough paths}

Almost all sample paths of a $d$-dimensional fractional Brownian motion (fBM) with Hurst parameter $H\in(\frac{1}{3},\frac{1}{2}]$ possess finite $p$-variation with $2<\frac{1}{H}<p<3$, and can be canonically enhanced to geometric rough paths. Coutin and Qian \cite{CQ02} constructed the canonical rough path enhancement $\mathbf{B}^{H,\text{Str}}=(1,B^{H},\mathbb{B}^{H,\text{Str}})$ in the Stratonovich sense using dyadic approximations of fBM and their iterated integrals. However, for the parameter estimation problem discussed in this paper, understanding the stochastic integral in the estimator (see Section \ref{sec-CRPE}) in the Stratonovich sense would almost surely result in convergence to 0, rendering the estimator unreasonable and useless. Therefore, we require a theory of It\^o-type integration (non-geometric rough path) for both fBM and fOU processes. In \cite{QX}, Qian and Xu constructed a
non-geometric rough path enhancement $\tilde{\mathbf{B}}^{H}=(1,B^{H},\tilde{\mathbb{B}}^{H})$
associated with an fBM by setting that 
\[
\tilde{\mathbb{B}}_{s,t}^{H}=\mathbb{B}_{s,t}^{H,\text{Str}}-\frac{1}{2}I(t^{2H}-s^{2H}),
\]
where $I$ denotes the $d\times d$ identity matrix. This construction of $\tilde{\mathbb{B}}_{s,t}^{H}$ allows for the definition of pathwise integrals with respect to the enhanced rough path $\tilde{\mathbf{B}}^{H}$. For the first level of this integral,
\[
\int F(B^{H})d_{\mathfrak{R}_{1}}\tilde{\mathbf{B}}^{H}(\omega)=\lim_{|\mathcal{P}|\to0}\sum_{[s,t]\in\mathcal{P}}F(B_{s}^{H}(\omega))B_{s,t}^{H}(\omega)+DF(B_{s}^{H}(\omega))\tilde{\mathbb{B}}_{s,t}^{H}(\omega)
\]
for every $\omega\in N^{c}$, where $N$ is a null set. The second
level is defined similarly as to (\ref{level-2}).

The theory of above rough path enhancement and associated It\^o integration is limited to one forms, meaning that it only works well for functions of $B_{t}^{H}$. Thus, this theory is not suitable for dealing with fOU processes, which depend on the whole path of ${B_{s}^{H},\ 0\leq s\leq t}$. In this paper, we reveal that a different integration theory is needed: one with different rough paths associated with fBM.

To define a non-geometric It\^o rough path enhancement associated with fBM suitable for the study of fOU processes, we take that $\varphi(t):=\frac{1}{2}It^{2H}-U(t)$, where
\begin{equation}
U(t):=H\Gamma\int_{0}^{t}\int_{0}^{s}e^{-\Gamma(s-u)}(s^{2H-1}-(s-u)^{2H-1})duds.
\end{equation}
This function has finite $q$-variation with $q=\frac{1}{2H}$, allowing us to define the  non-geometric It\^o type fractional Brownian rough path lift for $B^{H}$ as
\begin{equation}
\mathbf{B}_{s,t}^{H,\text{It\^o}}=(1,B_{s,t}^{H},\mathbb{B}_{s,t}^{H,\text{It\^o}}):=(1,B_{s,t}^{H},\mathbb{B}_{s,t}^{H,\text{Str}}-\varphi_{s,t}),
\end{equation}
where $\varphi_{s,t}=\varphi(t)-\varphi(s)$.

\begin{rem}\label{remark-ito} One can verify that, if $H=\frac{1}{2}$, this It\^o rough path enhancement is consistent with It\^o theory for the standard Brownian motion. When $\Gamma=0$, this enhancement is the same as the one form case defined in Qian and Xu \cite{QX}. In what follows, we will illustrate why we call
this It\^o rough path/It\^o rough integration. \end{rem}

\subsection{FOU as rough paths}

The fOU process $X_{t}$ defined by stochastic differential equation
(\ref{fOU-h}) can also be enhanced as a rough path according
to the theory of rough path, which is the essence of the theory of
rough differential equations. For the existence and uniqueness
of the solution to (\ref{fOU-h}), for this simple case through the theory of
rough path is not needed. However, when we ask if $X_{t}$ can be
enhanced to a rough path, or when we want to integrate $F(X)$ with
respect to $X$, the rough path analysis is a natural tool to deal
with these problems.

We emphasize that the meaning of the solution $\mathbf{X}$ to a rough differential
equation enhanced by (\ref{fOU-h}) depends on the rough
paths $\mathbf{B}^{H}$ we use. Here $\mathbf{B}^{H}$ can be either
$\mathbf{B}^{H,\text{Str}}$ (in Stratonovich sense) or $\mathbf{B}^{H,\text{It\^o}}$
(in It\^o sense).

Let $Z_{t}=(B_{t}^{H},X_{t},t)$ and $\mathbf{Z}=(\mathbf{B}^{H},\mathbf{X},\mathbf{t})$ be its associated rough path enhancement. Then equation (\ref{fOU-h})
is enhanced to
\begin{equation}
d_{\mathfrak{R}}\mathbf{Z}=f(Z)d_{\mathfrak{R}}\mathbf{Z},\label{RDE}
\end{equation}
where $f(x,y,t)(\xi,\eta,\tau):=(\xi,-\Gamma y\tau+\Sigma\xi,\tau)$.
According to Theorem 6.2.1 and Corollary 6.2.2 in \cite{LQ02}, a
unique solution $\mathbf{Z}$, which is a rough path, exists. Formally,
$\mathbf{Z}=(1,Z,\mathbb{Z})$ has the following expression:
\begin{align}
Z_{s,t} & =(B_{s,t}^{H},X_{s,t},t-s),\\
\mathbb{Z}_{s,t} & =\begin{pmatrix}\mathbb{B}_{s,t}^{H} & \int_{s}^{t}B_{s,u}^{H}dX_{u} & \int_{s}^{t}B_{s,u}^{H}du\\
\int_{s}^{t}X_{s,u}dB_{u}^{H} & \mathbb{X}_{s,t} & \int_{s}^{t}X_{s,u}du\\
\int_{s}^{t}(u-s)dB_{u}^{H} & \int_{s}^{t}(u-s)dX_{u} & \frac{1}{2}(t-s)^{2}
\end{pmatrix}.\label{Z2-RDE}
\end{align}
Each component of the second level $\mathbb{Z}_{s,t}$ is well-defined
as parts of the solution to (\ref{RDE}). More exactly, we denote
Stratonovich solution of RDE (\ref{RDE}) as $\mathbf{Z}^{\mathrm{Str}}=(1,Z,\mathbb{Z}^{\mathrm{Str}})$,
where $\mathbb{Z}^{\mathrm{Str}}=(\mathbb{Z}^{\mathrm{Str},ij})_{i,j=1,2,3}$,
and we denote that It\^o solution of RDE (\ref{RDE}) as $\mathbf{Z}^{\text{It\^o}}=(1,Z,\mathbb{Z}^{\text{It\^o}})$,
where $\mathbb{Z}^{\text{It\^o}}=(\mathbb{Z}^{\text{It\^o},ij})_{i,j=1,2,3}$.

We may therefore define that Stratonovich integral (first
level) of fOU process with respect to fBM as
\begin{equation}
\int_{0}^{t}X_{s}\circ d_{\mathfrak{R}_{1}}\mathbf{B}^{H,\mathrm{Str}}=\mathbb{Z}_{0,t}^{\mathrm{Str},21}+X_{0}B_{0,t}^{H},
\end{equation}
and It\^o integral (first level) of fOU process with respect to fBM
as
\begin{equation}
\int_{0}^{t}X_{s}d_{\mathfrak{R}_{1}}\mathbf{B}^{H,\text{It\^o}}=\mathbb{Z}_{0,t}^{\text{It\^o},21}+X_{0}B_{0,t}^{H}.
\end{equation}

Now we can define stochastic integrals with respect to the fOU rough path
enhancement $\mathbf{X}$ by Equations (\ref{level-1}) and (\ref{level-2}).
Note that these integrals are pathwise defined and continuous with
respect to the sample path $\mathbf{X}(\omega)$ in the $p$-variation
metric. In what follows, we denote the Stratonovich rough integral as
\[
\int_{0}^{t}F(X_{s})\circ d_{\mathfrak{R}}\mathbf{X}=\left(1,\int_{0}^{t}F(X_{s})\circ d_{\mathfrak{R}_{1}}\mathbf{X},\int_{0}^{t}F(X_{s})\circ d_{\mathfrak{R}_{2}}\mathbf{X}\right),
\]
and the It\^o rough integral as
\[
\int_{0}^{t}F(X_{s})d_{\mathfrak{R}}\mathbf{X}=\left(1,\int_{0}^{t}F(X_{s})d_{\mathfrak{R}_{1}}\mathbf{X},\int_{0}^{t}F(X_{s})d_{\mathfrak{R}_{2}}\mathbf{X}\right).
\]
As an application, we will use the It\^o rough integrals to construct
the estimator for parametric matrix $\Gamma$, and prove the asymptotic
properties and pathwise stability in the following sections.

\subsection{Zero expectation}

Let us illustrate the reason for the naming of It\^o rough paths and It\^o rough
integrals. In stochastic analysis, It\^o integrals can be defined in
terms of the martingale property, which is suitable for semi-martingales.
For processes which are not semi-martingales such as fBM, attempts making integrals with respect to fBM being martingales are of course
hopeless. We instead demand that the expectations of integrals with
respect to fBM are constant (e.g. to be zero). We call this kind
of integrals as an It\^o type integral, which is in fact an extension
of classical It\^o integration theory.

Now let us verify that expectation of the It\^o integral of fOU process
with respect to fBM $\int_{0}^{t}X_{s}d_{\mathfrak{R}_{1}}\mathbf{B}^{H,\text{It\^o}}$
(or write as $\int_{0}^{t}X_{s}\otimes d_{\mathfrak{R}_{1}}\mathbf{B}^{H,\text{It\^o}}$)
vanishes.

According to the theory of differential equations driven by rough
paths and the definition of integrals above, and assuming that coefficient
matrices $\Gamma$ and $\Sigma$ are commutative for simplicity, we
have that
\begin{equation}
\int_{0}^{t}X_{s}d_{\mathfrak{R}_{1}}\mathbf{B}^{H,\text{It\^o}}=\int_{0}^{t}X_{s}\circ d_{\mathfrak{R}_{1}}\mathbf{B}^{H,\mathrm{Str}}-\Sigma\varphi(t).
\end{equation}
since
\begin{equation}
X_{t}=e^{-\Gamma t}X_{0}+\int_{0}^{t}e^{-\Gamma(t-s)}\Sigma dB_{s}^{H},
\end{equation}
where $X_{0}$ is a constant vector and the integral on the right
hand side is Young's integral and equals $\int_{0}^{t}e^{-\Gamma(t-s)}\Sigma\circ d_{\mathfrak{R}_{1}}\mathbf{B}_{s}^{H,\mathrm{Str}}$.
Therefore
\[
\begin{split}\mathbb{E}\left(\int_{0}^{t}X_{s}\circ d_{\mathfrak{R}_{1}}\mathbf{B}^{H,\mathrm{Str}}\right) & =\mathbb{E}\left(\int_{0}^{t}e^{-\Gamma s}X_{0}\circ d_{\mathfrak{R}_{1}}\mathbf{B}^{H,\mathrm{Str}}\right)\\
 & +\mathbb{E}\left(\int_{0}^{t}\int_{0}^{s}e^{-\Gamma(s-u)}\Sigma\circ d_{\mathfrak{R}_{1}}\mathbf{B}_{u}^{H,\mathrm{Str}}\circ d_{\mathfrak{R}_{1}}\mathbf{B}_{s}^{H,\mathrm{Str}}\right).
\end{split}
\]
The first term on the right hand side is zero, and the second term is
\[
\begin{split} & \mathbb{E}\left(\int_{0}^{t}\int_{0}^{s}e^{-\Gamma(s-u)}\Sigma\circ d_{\mathfrak{R}_{1}}\mathbf{B}_{u}^{H,\mathrm{Str}}\circ d_{\mathfrak{R}_{1}}\mathbf{B}_{s}^{H,\mathrm{Str}}\right)=\int_{0}^{t}\int_{0}^{s}e^{-\Gamma(s-u)}\Sigma dR_{H}(u,s)\\
 & \ \ \ \ =\Sigma\left(\frac{1}{2}It^{2H}-H\Gamma\int_{0}^{t}\int_{0}^{s}e^{-\Gamma(s-u)}(s^{2H-1}-(s-u)^{2H-1})duds\right)=\Sigma\varphi(t),
\end{split}
\]
where $R_{H}(u,s)=\mathbb{E}(B_{u}^{H}B_{s}^{H})=\frac{1}{2}(u^{2H}+s^{2H}-|u-s|^{2H})$
is the covariance function of fBM and the integral against $R_{H}(u,s)$
is defined as a Young's integral in the 2D sense (see e.g. \cite{FV10}).
Thus, combining equations above, we have prove the zero expectation
property, i.e.,
\begin{equation}
\mathbb{E}\left(\int_{0}^{t}X_{s}d_{\mathfrak{R}_{1}}\mathbf{B}^{H,\text{It\^o}}\right)=0.
\end{equation}

\section{Long time asymptotic of L\'evy area of fOU processes}

\label{sec-longtime}

In this section, we examine the properties of fOU processes. We establish the $\alpha$-H\"older continuity of fOU processes and prove a long-term asymptotic property of the L\'evy area of fOU processes.

\subsection{Regularity of fOU processes}

\subsubsection{The covariance of fOU processes}

The covariance function of a general fOU process can be worked out
explicitly. For simplicity, we first study a stationary version of the fOU process in this section. Consider
\[
X_{t}=\sigma\int_{-\infty}^{t}e^{-\lambda(t-s)}dB_{s}^{H},
\]
which is stationary and ergodic (see e.g. \cite{CKM03}), and that $B^{H}$
is fBM with Hurst parameter $H<\frac{1}{2}$. It is well known that
the covariance $R_{H}(\cdot,\cdot)$ of $B^{H}$ is of finite $\frac{1}{2H}$-variation.

The covariance function of $\{X_{t}=\sigma\int_{-\infty}^{t}e^{-\lambda(t-s)}dB_{s}^{H},t\geq0\}$
is given by (see, e.g. \cite{PT00})
\[
\begin{split}r(t) & =\mathrm{Cov}(X_{s},X_{s+t})=\mathrm{Cov}(X_{0},X_{t})\\
 & =\frac{\sigma^{2}}{\lambda^{2H}}\frac{G(2H+1)\sin(\pi H)}{\pi}\int_{0}^{\infty}\cos(\lambda tx)\frac{x^{1-2H}}{1+x^{2}}dx\\
 & =\frac{\sigma^{2}}{2\lambda^{2H}}G(2H+1)\cosh(\lambda t)-\frac{\sigma^{2}}{2}t^{2H}{}_{1}F_{2}(1;H+\frac{1}{2},H+1;\frac{1}{4}\lambda^{2}t^{2}),
\end{split}
\]
where $G(\cdot)$ is the Gamma function, $\cosh(\cdot)$ the hyperbolic
cosine function, $_{1}F_{2}(\cdot;\cdot,\cdot;\cdot)$ the generalized
hypergeometric function, i.e.,
\[
_{p}F_{q}(a_{1},\cdots,a_{p};b_{1},\cdots,b_{q};x)=\sum_{n=0}^{\infty}\frac{(a_{1})_{n}\cdots(a_{p})_{n}}{(b_{1})_{n}\cdots(b_{q})_{n}}\frac{x^{n}}{n!},
\]
and $(a)_{0}=1,(a)_{n}=a(a+1)\cdots(a+n-1)$, for $n\geq1$. One can
see the figure of this covariance function $r(\cdot)$ and its first
and second derivatives below, where we take $H=0.2$, $\sigma=\lambda=1$
as an example.

\begin{figure}
\begin{centering}
\includegraphics[width=12cm,height=8cm]{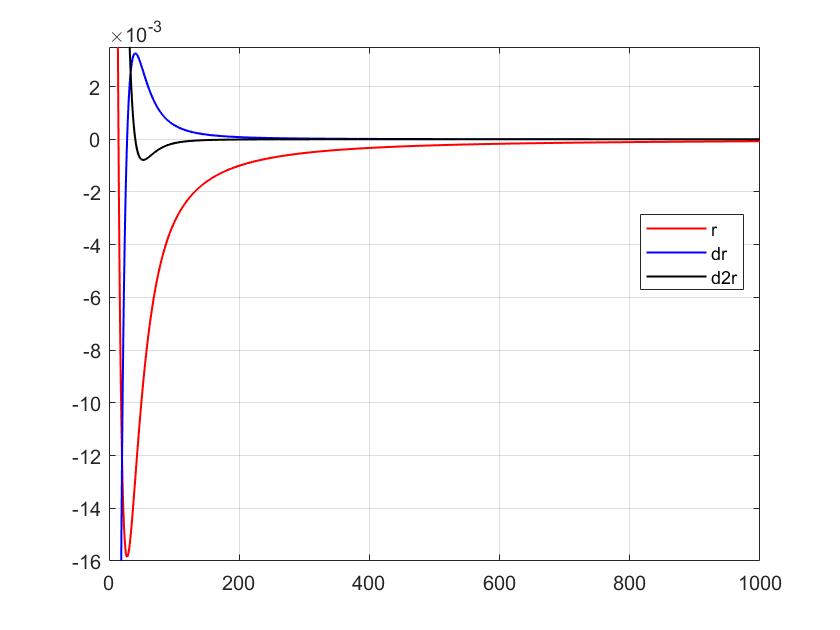}
\par\end{centering}
\caption{Graph of the covariance function $r(\cdot)$ of stationary fOU and
its first two derivatives, e.g. $H=0.2$, $\sigma=\lambda=1$.}
\label{fig-1}
\end{figure}

\begin{lem}\label{corr-lemma-r} For the covariance function $r(\cdot)$
of the stationary fOU process $X$ with $H<\frac{1}{2}$, we have the
following properties:\\
 (i) $r(t)$ is $2H$-H\"{o}lder continuous on $\mathbb{R}_{+}$, that
is,
\[
|r(t)-r(s)|\leq C_{H}|t-s|^{2H}
\]
for any $s,t\in\mathbb{R}_{+}$ and $C_{H}$ depends on $H,\sigma,\lambda$
only (we may ignore $\sigma,\lambda$).\\
 (ii) There exist constants $0<T_{0}<T_{1}$ such that $r''(T_{0})=0$, $r''(t)>0$ on interval $(0,T_{0})$, and $r''(t)<0$ on interval
$(T_{1},\infty)$. That is, $r$ is convex on $[0,T_{0}]$ and concave
on $(T_{1},\infty)$. \end{lem} \begin{proof} 
For covariance function $r(\cdot)$, near $t=0$,
\[
r(t)=\sigma^{2}\lambda^{-2H}HG(2H)\left(1-\frac{\lambda^{2H}}{G(2H+1)}t^{2H}+o(t^{2H})\right),
\]
and for $t$ large enough (see Theorem 2.3, \cite{CKM03}),
\[
r(t)=\frac{1}{2}\sigma^{2}\sum_{n=1}^{N}\lambda^{-2n}\left(\prod_{k=0}^{2n-1}(2H-k)\right)t^{2H-2n}+O(t^{2H-2N-2}).
\]
Since $r(t)$ is continuous on $[0,\infty)$, one can also see
that $r(t)$ has polynomial decay to zero as $t$ is large from above
equality.

For (i), we have that $\max_{t\geq0}|r(t)|=C<\infty$, for any $s,t\in\mathbb{R}_{+}$, 
and $|t-s|\geq1$, so
\[
|r(t)-r(s)|\leq2C\leq2C|t-s|^{2H}.
\]
For any $s,t\in\mathbb{R}_{+}$ and $|t-s|<1$, we show the statement
in three cases: $s,t\in[0,1]$, $s,t\in[1,\infty)$ and $0\leq s<1<t(<2)$.
For the first and third terms, we actually need to show that, for any
$s,t\in[0,2]$ and $|t-s|<1$, there exists a constant $C$ such that
$|r(t)-r(s)|\leq C|t-s|^{2H}.$ Since $r(t)=-ct^{2H}+\varphi(t)$,
where $\varphi(\cdot)$ is smooth on $\mathbb{R}_{+}$, 
\[
|r(t)-r(s)|\leq c|t-s|^{2H}+\max_{0\leq u\leq2}|\varphi'(u)||t-s|\leq C|t-s|^{2H}.
\]
For the second case, i.e., for any $s,t\in[1,\infty)$ and $|t-s|<1$,
we have
\[
|r(t)-r(s)|\leq\max_{u\geq1}|r'(u)||t-s|\leq C|t-s|^{2H}.
\]
Thus we have prove the statement (i).

For (ii), one can see that there exists a small number $\varepsilon>0$
such that $r''(\varepsilon)>0$ and a large number $T_{1}>0$ such
that, for all $t\geq T_{1}$, $r''(t)<0$. By the continuity of $r''$
on $(0,\infty)$, there exists a $T_{0}\in(\varepsilon,T_{1})$ satisfying that 
$r''(T_{0})=0$ and $r''(t)>0$ for $t\in(0,T_{0})$. \end{proof}

What follows are important properties of fOU processes when $H<\frac{1}{2}$.
It is well-known that, increments of fBM are negatively correlated
when $H<\frac{1}{2}$, and positively correlated when $H>\frac{1}{2}$,
while for $H=\frac{1}{2}$, increments over different time periods
are independent. We found that for fOU process with $H<\frac{1}{2}$,
the disjoint increments are \textit{locally negative correlated}.
If the distance of the intervals corresponding to the disjoint increments
is large, then they are positively correlated, and we call this \textit{long-range
positive correlation} (look at the theorem below). Heuristically, fOU process
is locally like fBM so that it has the locally negative correlation
property as fBM when $H<\frac{1}{2}$. For long distance the drift
becomes the dominated force, so the fOU behaves positively correlated.
In the case where $H=\frac{1}{2}$, the fOU is the standard OU process
driven by standard Bownian motion. The properties of it are well known.
Our main concern here is for the true fOU process case with $H<\frac{1}{2}$.
\begin{thm}\label{lem-correlation} Consider the stationary fOU process
$X$ with $H<\frac{1}{2}$. $T_{0},T_{1}$ are given in the previous
lemma. \\
 (i) (Locally negative correlation) For any $s_{0}$ and $s_{0}\leq t_{i}<t_{i+1}\leq t_{j}<t_{j+1}\leq s_{0}+T_{0}$,
then
\begin{equation}
\mathbb{E}(X_{t_{i+1}}-X_{t_{i}})(X_{t_{j+1}}-X_{t_{j}})\leq0.\label{negative-correlated}
\end{equation}
(ii) (Long-range positive correlation) For any $0\leq t_{i}<t_{i+1}<t_{j}<t_{j+1}$,
and if $t_{j}-t_{i+1}>T_{1}$, then
\begin{equation}
\mathbb{E}(X_{t_{i+1}}-X_{t_{i}})(X_{t_{j+1}}-X_{t_{j}})\geq0.\label{positive-correlated}
\end{equation}
\end{thm} \begin{proof} (i) Since
\[
\begin{split} & \ \ \ \ \mathbb{E}(X_{t_{i+1}}-X_{t_{i}})(X_{t_{j+1}}-X_{t_{j}})\\
 & =\left(r(t_{j+1}-t_{i+1})-r(t_{j+1}-t_{i})\right)-\left(r(t_{j}-t_{i+1})-r(t_{j}-t_{i})\right)\\
 & =\left(r(x_{3})-r(x_{4})\right)-\left(r(x_{1})-r(x_{2})\right)\\
 & =-\left[\left(r(x_{4})-r(x_{3})\right)-\left(r(x_{2})-r(x_{1})\right)\right],
\end{split}
\]
where $x_{1}:=t_{j}-t_{i+1}$, $x_{2}:=t_{j}-t_{i}$, $x_{3}:=t_{j+1}-t_{i+1}$,
$x_{4}:=t_{j+1}-t_{i}$, we have that $0\leq x_{1}<x_{2}\leq x_{3}<x_{4}\leq T_{0}$
or $0\leq x_{1}<x_{3}\leq x_{2}<x_{4}\leq T_{0}$, and
\[
\frac{r(x_{4})-r(x_{3})}{x_{4}-x_{3}}\geq\frac{r(x_{2})-r(x_{1})}{x_{2}-x_{1}},
\]
by the convexity of $r$. This proves (\ref{negative-correlated}). \\
 (ii) The proof of (\ref{positive-correlated}) is almost the same
as (i). \end{proof}

Recall that, the stationary fOU process $X$ satisfies the inequality
\begin{equation}
\mathbb{E}|X_{t}-X_{s}|^{2}\leq C_{H}|t-s|^{2H}
\end{equation}
for any $s,t\in\mathbb{R}_{+}$, and $C_{H}$ depends on $H,\sigma,\lambda$ (we ignore $\sigma,\lambda$ in the notation here). This result can be found in Lemma 5.7 of \cite{KMR2017}.


\begin{prop}\label{Variation} Let $X$ be the stationary fOU process
with $H\in(0,\frac{1}{2})$. Then its covariance $R_{X}(s,t)=\mathbb{E}(X_{s}X_{t})$
is of finite $\frac{1}{2H}$-variation on $[s_{0},s_{0}+T_{0}]^{2}$
in the $2D$ sense for any $s_{0}$. Moreover, there exist constants $C=C(H)$
and $T_{0}>0$ such that, for all $s<t$ in $[s_{0},s_{0}+T_{0}]$,
\begin{equation}
|R_{X}|_{\frac{1}{2H}-var;[s,t]^{2}}\leq C(H)|t-s|^{2H},
\end{equation}
where
\begin{equation}
|R_{X}|_{\rho-var;[s,t]^{2}}^{\rho}:=\sup\sum_{i,j}\left|\mathbb{E}\left[(X_{t_{i+1}}-X_{t_{i}})(X_{t'_{j+1}}-X_{t'_{j}})\right]\right|^{\rho},
\end{equation}
and $\mathcal{P}=\{t_{i}\}$, $\mathcal{P}'=\{t'_{j}\}$ are any two
partitions of interval $[s,t]$. \end{prop} \begin{proof} By Lemma
5.54 of \cite{FV10}, we just need to show the finite $\frac{1}{2H}$-variation
by the same partition $\mathcal{P}=\{t_{i}\}$ of interval $[s,t]\subset[s_{0},s_{0}+T_{0}]$.
Let us consider that
\begin{equation}
\sum_{i,j}\left|\mathbb{E}\left[(X_{t_{i+1}}-X_{t_{i}})(X_{t_{j+1}}-X_{t_{j}})\right]\right|^{\frac{1}{2H}}.
\end{equation}
For a fixed $i$, and $i\neq j$, $\mathbb{E}\left[(X_{t_{i+1}}-X_{t_{i}})(X_{t_{j+1}}-X_{t_{j}})\right]\leq0$
for $H<\frac{1}{2}$ by Theorem \ref{lem-correlation}, and hence,
\[
\begin{split} & \ \ \ \ \sum_{j}\left|\mathbb{E}\left[X_{t_{i},t_{i+1}}X_{t_{j},t_{j+1}}\right]\right|^{\frac{1}{2H}}\\
 & =\sum_{j\neq i}\left|\mathbb{E}\left[X_{t_{i},t_{i+1}}X_{t_{j},t_{j+1}}\right]\right|^{\frac{1}{2H}}+\left(\mathbb{E}\left|X_{t_{i},t_{i+1}}\right|^{2}\right)^{\frac{1}{2H}}\\
 & \leq\left|\mathbb{E}\left(\sum_{j\neq i}X_{t_{i},t_{i+1}}X_{t_{j},t_{j+1}}\right)\right|^{\frac{1}{2H}}+\left(\mathbb{E}\left|X_{t_{i},t_{i+1}}\right|^{2}\right)^{\frac{1}{2H}}\\
 & \leq\left(2^{\frac{1}{2H}-1}\left|\mathbb{E}\left(\sum_{j}X_{t_{i},t_{i+1}}X_{t_{j},t_{j+1}}\right)\right|^{\frac{1}{2H}}+2^{\frac{1}{2H}-1}\left(\mathbb{E}\left|X_{t_{i},t_{i+1}}\right|^{2}\right)^{\frac{1}{2H}}\right)\\
 & \ \ \ \ +\left(\mathbb{E}\left|X_{t_{i},t_{i+1}}\right|^{2}\right)^{\frac{1}{2H}}\\
 & \leq C(H)\left|\mathbb{E}\left[X_{t_{i},t_{i+1}}X_{s,t}\right]\right|^{\frac{1}{2H}}+C(H)\left(\mathbb{E}\left|X_{t_{i},t_{i+1}}\right|^{2}\right)^{\frac{1}{2H}}.
\end{split}
\]
Therefore, we have that 
\[
\sum_{i,j}\left|\mathbb{E}\left[X_{t_{i},t_{i+1}}X_{t_{j},t_{j+1}}\right]\right|^{\frac{1}{2H}}\leq C(H)\sum_{i}\left|\mathbb{E}\left[X_{t_{i},t_{i+1}}X_{s,t}\right]\right|^{\frac{1}{2H}}+C(H)\sum_{i}\left(\mathbb{E}\left|X_{t_{i},t_{i+1}}\right|^{2}\right)^{\frac{1}{2H}}.
\]
The second term on the right hand side is controlled by $C(H)|t-s|$. Now we show that
\[
\sum_{i}\left|\mathbb{E}\left[X_{t_{i},t_{i+1}}X_{s,t}\right]\right|^{\frac{1}{2H}}\leq C(H)|t-s|.
\]
Since
\[
\begin{split}\left|\mathbb{E}\left[X_{t_{i},t_{i+1}}X_{s,t}\right]\right| & =\left|\mathbb{E}\left(X_{t_{i+1}}X_{t}-X_{t_{i}}X_{t}+X_{t_{i}}X_{s}-X_{t_{i+1}}X_{s}\right)\right|\\
 & =\left|r(t-t_{i+1})-r(t-t_{i})+r(t_{i}-s)-r(t_{i+1}-s)\right|\\
 & \leq\left|r(t-t_{i+1})-r(t-t_{i})\right|+\left|r(t_{i}-s)-r(t_{i+1}-s)\right|\\
 & \leq C_{H}\left|t_{i+1}-t_{i}\right|^{2H}+C_{H}\left|t_{i+1}-t_{i}\right|^{2H}\leq2C_{H}\left|t_{i+1}-t_{i}\right|^{2H},
\end{split}
\]
we have that
\[
\sum_{i}\left|\mathbb{E}\left[X_{t_{i},t_{i+1}}X_{s,t}\right]\right|^{\frac{1}{2H}}\leq\sum_{i}C(H)\left|t_{i+1}-t_{i}\right|\leq C(H)|t-s|.
\]
Now we have completed the proof. \end{proof}

\begin{cor}\label{Variation-cor} Let $X$ be the stationary fOU
process with $H\in(0,\frac{1}{2})$. Then its covariance $R_{X}(s,t)=\mathbb{E}(X_{s}X_{t})$
is of finite $\frac{1}{2H}$-variation on $[0,T]^{2}$ in the $2D$ sense.
Moreover, there exists a constant $C=C(H)$ such that, for all $s<t$
in $[0,T]$,
\begin{equation}
|R_{X}|_{\frac{1}{2H}-var;[s,t]^{2}}^{\frac{1}{2H}}\leq C(H)|t-s|.
\end{equation}
\end{cor} \begin{proof} We divide the interval $[0,T]$ into $m+1=\left[\frac{T}{T_{0}}\right]+1$
pieces, and denote them as $[0,T_{0}]$, $[T_{0},2T_{0}]$, $\cdots$,
$[(m-1)T_{0},mT_{0}]$, $[mT_{0},T]$. For any subinterval $[s,t]\subset[0,T]$,
there exist $q_{1},q_{2}\in\mathbb{N}$ such that $s\in[(q_{1}-1)T_{0},q_{1}T_{0}]$
and $t\in[q_{2}T_{0},(q_{2}+1)T_{0}]$. By the subadditivity of $|R_{X}|_{\frac{1}{2H}-var;[\cdot,\cdot]^{2}}^{\frac{1}{2H}}$, we then have that 
\[
\begin{split}|R_{X}|_{\frac{1}{2H}-var;[s,t]^{2}}^{\frac{1}{2H}} & \leq|R_{X}|_{\frac{1}{2H}-var;[s,q_{1}T_{0}]^{2}}^{\frac{1}{2H}}+|R_{X}|_{\frac{1}{2H}-var;[q_{1}T_{0},(q_{1}+1)T_{0}]^{2}}^{\frac{1}{2H}}+\cdots+|R_{X}|_{\frac{1}{2H}-var;[q_{2}T_{0},t]^{2}}^{\frac{1}{2H}}\\
 & \leq C(H)(|q_{1}T_{0}-s|+|2q_{1}T_{0}-q_{1}T_{0}|+\cdots+|t-q_{2}T_{0}|)\\
 & \leq C(H)|t-s|.
\end{split}
\]
This completes the proof of the corollary. \end{proof}

\subsubsection{Regularity of fOU processes}

In what follows, we study the $\alpha$-H\"older continuity of the one
dimensional, stationary fOU process $X_{t}=\sigma\int_{-\infty}^{t}e^{-\lambda(t-s)}dB_{s}^{H}$.
Before showing the regularity, we recall the usual Garsia-Rodemich-Rumsey
inequality (see e.g., page 60, Stroock and Varadhan \cite{SV79}).
\begin{lem}(Garsia-Rodemich-Rumsey inequality) Let $p(\cdot)$ and
$\Psi(\cdot)$ be continuous, strictly increasing functions on $[0,\infty)$
such that
\[
p(0)=\Psi(0)=0,\ \ \text{and}\ \ \lim_{t\to\infty}\Psi(t)=\infty.
\]
Given $T>0$ and $\phi\in C([0,T],\mathbb{R}^{d})$, if there is a
constant $B$ such that
\begin{equation}
\int_{0}^{T}\int_{0}^{T}\Psi\left(\frac{|\phi(t)-\phi(s)|}{p(|t-s|)}\right)dsdt\leq B,
\end{equation}
then, for all $0\leq s\leq t\leq T$,
\begin{equation}
|\phi(t)-\phi(s)|\leq8\int_{0}^{|t-s|}\Psi^{-1}\left(\frac{4B}{u^{2}}\right)p(du).
\end{equation}
\end{lem}

As an application of the lemma above, we have

\begin{prop}\label{Holder-X-lem} Let $X$ be a one dimensional,
stationary fOU process with $H\in(0,\frac{1}{2})$ on $[0,T]$. Then
there exist a constant $0<\beta<1$ and an almost surely finite random
variable $C$ independent of $T$ such that
\begin{equation}
|X_{t}-X_{s}|\leq CT^{\beta}|t-s|^{\alpha},\ a.s.\label{Holder-X}
\end{equation}
for any $\alpha\in(0,H)$, and any $0\leq s,t\leq T$.
\end{prop}
\begin{proof}
Recall that
\begin{equation}
\mathbb{E}|X_{t}-X_{s}|^{2}\leq C_{H}|t-s|^{2H}.
\end{equation}
Since $X_{t}$ is Gaussian process, all the norms are equivalent,
and we get that 
\begin{equation}
\mathbb{E}|X_{t}-X_{s}|^{p}\leq C_{p}(\mathbb{E}|X_{t}-X_{s}|^{2})^{\frac{p}{2}}\leq C_{p,H}|t-s|^{pH}\label{pth_moment_X}
\end{equation}
for any $p>2$.

Next, we apply the Garsia-Rodemich-Rumsey inequality. Take $\Psi(x)=x^{p}$
and $p(x)=x^{H}$. Then inequality (\ref{pth_moment_X}) implies that
\[
\mathbb{E}\left(\int_{0}^{T}\int_{0}^{T}\Psi\left(\frac{|X_{t}-X_{s}|}{p(|t-s|)}\right)dsdt\right)\leq C_{p,H}T^{2}.
\]
Define that
\[
B_{T}:=\int_{0}^{T}\int_{0}^{T}\Psi\left(\frac{|X_{t}-X_{s}|}{p(|t-s|)}\right)dsdt=\int_{0}^{T}\int_{0}^{T}\frac{|X_{t}-X_{s}|^{p}}{|t-s|^{pH}}dsdt.
\]
Then, for any $q>3$, we get that 
\[
\mathbb{E}\left(\sum_{n=1}^{\infty}\frac{B_{n}}{n^{q}}\right)=\sum_{n=1}^{\infty}\frac{\mathbb{E}(B_{n})}{n^{q}}\leq\sum_{n=1}^{\infty}\frac{Cn^{2}}{n^{q}}<\infty.
\]
Thus there exists an almost surely finite random variable $R$ independent
of $n$ such that
\[
\sum_{n=1}^{\infty}\frac{B_{n}}{n^{q}}\leq R,\ a.s.,
\]
we have
\[
B_{n}\leq Rn^{q},\ a.s.\ \ \ \forall n\geq1,\ q>3.
\]
Take $n=[T]$. Then
\[
B_{T}\leq B_{n+1}\leq R(n+1)^{q}\leq CRT^{q},\ a.s.\ \ \ \forall T>0,\ q>3.
\]
Then the Garsia-Rodemich-Rumsey inequality gives that
\[
\begin{split}\left|X_{t}-X_{s}\right| & \leq8\int_{0}^{|t-s|}\Psi^{-1}\left(\frac{4B_{T}}{u^{2}}\right)p(du)\\
 & \leq C(4B_{T})^{\frac{1}{p}}|t-s|^{H-2/p}\leq CR^{\frac{1}{p}}T^{\frac{q}{p}}|t-s|^{\alpha}
\end{split}
\]
for any $\alpha\in(0,H)$, $p>3\vee\left[\frac{2}{H-\alpha}\right]$
and $3<q<p$. This concludes the lemma. \end{proof}

\begin{rem} When $X_{t}=\sigma\int_{0}^{t}e^{-\lambda(t-s)}dB_{s}^{H}$,
it still satisfies inequality (\ref{Holder-X}). \end{rem}

Additionally, we prove a proposition for a function of fOU processes, which will be applied in Section \ref{sec-DRPE}. Here, we introduce the process $\overline{X}_{t}^{i}$ as follows:
\begin{equation}
\overline{X}_{t}^{i}=\sigma\int_{-\infty}^{t}e^{-\lambda_{i}(t-s)}dB_{s}^{H,i},\ i=1,2,\cdots,d.\label{fOU-indept-stationary}
\end{equation}
It is worth noting that the processes $\{\overline{X}_{t}^{i},t\geq0\}$ are stationary, ergodic, Gaussian
processes, as discussed in \cite{CKM03}.

\begin{prop}\label{Holder-F(X)-lem}
Let $B^{H}=(B^{H,1},B^{H,2},\cdots,B^{H,d})$ be a $d$-dimensional
fBM with $H\in(0,\frac{1}{2})$, and $X=(X^{1},X^{2},\cdots,X^{d})$ a
$d$-dimensional fOU process, where $X_{t}^{i}=\sigma\int_{0}^{t}e^{-\lambda_{i}(t-s)}dB_{s}^{H,i}$,
$\lambda_{i}>0$, $\sigma\in\mathbb{R}$. Define that $F(X_{t}):=X_{t}\otimes X_{t}=(X_{t}^{i}X_{t}^{j})_{i,j=1,2,\cdots,d}$,
and the norm of matrix $A$ as $\|A\|=\sum_{i,j=1}^{d}|a_{ij}|$.
Then there exist a constant $0<\beta<1$, an almost surely finite
random variable $C$ (independent of $T$), and a random variable $R_{T}$
(tends to zero almost surely as $T\to\infty$) such that
\begin{equation}
\sup_{s\neq t}\frac{\left\Vert F(X_{t})-F(X_{s})\right\Vert }{|t-s|^{\alpha}}\leq CR_{T}T^{\beta},\label{Holder-F(X)}
\end{equation}
for any $0\leq s,t\leq T$, and any $\alpha\in(0,H)$. \end{prop} \begin{proof}
First, we present a fact about supremum of the one dimensional, stationary
fOU process $(\overline{X}^{i})_{t}^{*}:=\sup_{0\leq s\leq t}|\overline{X}_{s}^{i}|$
below. Since we know that $\overline{X}^{i}$ and $-\overline{X}^{i}$
have the same distribution, and their covariance function is
\begin{equation}
r_{i}(t)=\mathrm{Cov}(\overline{X}_{s+t}^{i},\overline{X}_{s}^{i})=C\left(1-\frac{\lambda_{i}^{2H}}{G(2H+1)}t^{2H}+o(t^{2H})\right)
\end{equation}
for $t$ small, where $C=\sigma^{2}\lambda_{i}^{-2H}HG(2H)$ and $G(\cdot)$
is Gamma function, by Theorem 3.1 of Pickands \cite{Pick69}, we
know that, for $t$ tending to infinity
\[
\frac{1}{t^{\delta}}\sup_{0\leq s\leq t}\overline{X}_{s}^{i}\to0,\ a.s.,\ \ \frac{1}{t^{\delta}}\sup_{0\leq s\leq t}(-\overline{X}_{s}^{i})\to0,\ a.s.
\]
for any $\delta>0$. Since $(\overline{X}^{i})_{t}^{*}=(\sup_{0\leq s\leq t}\overline{X}_{s}^{i})\vee(\sup_{0\leq s\leq t}(-\overline{X}_{s}^{i}))$,
then
\begin{equation}
\frac{(\overline{X}^{i})_{t}^{*}}{t^{\delta}}\to0,\ a.s.\label{ineq-Pickand}
\end{equation}
Since $X_{t}^{i}=\overline{X}_{t}^{i}-e^{-\lambda t}\overline{X}_{0}^{i}$,
we also have that $\frac{(X^{i})_{t}^{*}}{t^{\delta}}\to0\ a.s.$, where
$(X^{i})_{t}^{*}:=\sup_{0\leq s\leq t}|X_{s}^{i}|$.

Now define $R_{t}=\sup_{i=1,\cdots,d}\frac{(X^{i})_{t}^{*}}{t^{\delta}}$,
then $R_{t}\to0\ a.s.$ as $t\to\infty$. For any $i,j=1,2,\cdots,d$,
and $0\leq s,t\leq T$,
\[
\begin{split}|X_{t}^{i}X_{t}^{j}-X_{s}^{i}X_{s}^{j}| & =|(X_{t}^{i}-X_{s}^{i})X_{t}^{j}+X_{s}^{i}(X_{t}^{j}-X_{s}^{j})|\\
 & \leq|X_{t}^{j}||X_{t}^{i}-X_{s}^{i}|+|X_{s}^{i}||X_{t}^{j}-X_{s}^{j}|\\
 & \leq(X^{j})_{T}^{*}|X_{t}^{i}-X_{s}^{i}|+(X^{i})_{T}^{*}|X_{t}^{j}-X_{s}^{j}|\\
 & \leq CR_{T}T^{\delta}T^{\beta}|t-s|^{\alpha}+CR_{T}T^{\delta}T^{\beta}|t-s|^{\alpha}\\
 & \leq CR_{T}T^{\delta+\beta}|t-s|^{\alpha},
\end{split}
\]
where the second last inequality follows from Proposition \ref{Holder-X-lem}.
One can choose $\delta,\beta$ such that $0<\delta+\beta=:\beta'<1$. This completes the proof of the statement. \end{proof}

\subsubsection{L\'evy area of multi-dimensional fOU processes}

\label{holder-levy-subsec} In this subsection, let $B^{H}=(B^{H,1},B^{H,2},\cdots,B^{H,d})$
be a $d$-dimensional fBM with $H\in(\frac{1}{3},\frac{1}{2})$, and $X=(X^{1},X^{2},\cdots,X^{d})$
a $d$-dimensional fOU process, where $X_{t}^{i}=\sigma\int_{-\infty}^{t}e^{-\lambda_{i}(t-s)}dB_{s}^{H,i}$,
$\lambda_{i}>0$, $\sigma\in\mathbb{R}$. Then $X=(X^{1},X^{2},\cdots,X^{d})$
is stationary (see \cite{CKM03}), and covariance function is given
by
\[
R_{X}(s,t)=\mathrm{diag}(R_{1}(s,t),\cdots,R_{d}(s,t)),
\]
where $R_{i}(s,t)=\mathbb{E}(X_{s}^{i}X_{t}^{i})$.

In this subsection, we will show one estimate for off-diagonal elements
of the L\'evy area $\int_{0}^{t}X_{u}^{i}\circ d_{\mathfrak{R}_{1}}\mathbf{X}^{j}$
of the multi-dimensional fOU process $X$. We denote Stratonovich's
L\'evy area of the fOU process $X$ as
\[
A(t):=\int_{0}^{t}X_{u}\circ d_{\mathfrak{R}_{1}}\mathbf{X}=\left(\int_{0}^{t}X_{u}^{i}\circ d_{\mathfrak{R}_{1}}\mathbf{X}^{j}\right)_{i,j=1,2,\cdots,d},
\]
and $A_{ij}(t)$ as its components.


Before showing the estimate of off-diagonal elements, we recall a
lemma based on Wiener chaos. We denote $\mathcal{H}_{n}(\mathbb{P})$
as the homogeneous Wiener chaos of order $n$ and $\mathcal{C}^{n}(\mathbb{P}):=\oplus_{j=0}^{n}\mathcal{H}_{j}(\mathbb{P})$
the Wiener chaos (or non-homogeneous chaos) of order $n$. The lemma
below gives the hypercontractivity of Wiener chaos.

\begin{lem}\label{hypercontract}(Refer to, e.g., Lemma 15.21, \cite{FV10})
Let $q\in\mathbb{N}$ and $Z\in\mathcal{C}^{q}(\mathbb{P})$. Then,
for $p>2$,
\begin{equation}
(\mathbb{E}|Z|^{2})^{\frac{1}{2}}\leq(\mathbb{E}|Z|^{p})^{\frac{1}{p}}\leq(q+1)(p-1)^{\frac{q}{2}}(\mathbb{E}|Z|^{2})^{\frac{1}{2}}.
\end{equation}
\end{lem}

Now we illustrate one estimate for off-diagonal elements, i.e., when
$i\neq j$, we have  
\begin{prop}\label{Holder-prop}
Let $X=(X^{1},\cdots,X^{d})$ be a $d$-dimensional, stationary fOU
process with $H\in(\frac{1}{3},\frac{1}{2})$, and let $A_{ij}(t)=\int_{0}^{t}X_{u}^{i}\circ d_{\mathfrak{R}_{1}}\mathbf{X}^{j},i\neq j$
be the off-diagonal elements of Stratonovich's L\'evy area of $X$.
Then there exist $0<\beta<1$ and an almost surely finite random variable
$\widetilde{C}$ such that
\begin{equation}
|A_{ij}(t)-A_{ij}(s)|\leq\widetilde{C}n^{\beta},\ a.s.
\end{equation}
for any $s,t\in[n-1,n]$ and any integer $n\geq1$. \end{prop} \begin{proof}
First, we rewrite $R_{i}(s,t)$ as
\[
R_{i}\binom{\ s\ }{\ t\ }=\mathbb{E}X_{s}^{i}X_{t}^{i},
\]
and denote that 
\[
R_{i}\binom{s}{u,v}=\mathbb{E}X_{s}^{i}X_{u,v}^{i},\ R_{i}\binom{s,t}{u}=\mathbb{E}X_{s,t}^{i}X_{u}^{i},\ R_{i}\binom{s,t}{u,v}=\mathbb{E}X_{s,t}^{i}X_{u,v}^{i}.
\]
For the second moment of the L\'evy area,
\[
\begin{split}\mathbb{E}\left(\left|\int_{s}^{t}X_{u}^{i}\circ d_{\mathfrak{R}_{1}}\mathbf{X}^{j}\right|^{2}\right) & =\mathbb{E}\left(\int_{s}^{t}\int_{s}^{t}X_{u}^{i}X_{v}^{i}\circ d_{\mathfrak{R}_{1}}\mathbf{X}^{j}\circ d_{\mathfrak{R}_{1}}\mathbf{X}^{j}\right)\\
 & =\int_{s}^{t}\int_{s}^{t}\mathbb{E}(X_{u}^{i}X_{v}^{i})d\mathbb{E}(X_{u}^{j}X_{v}^{j})\\
 & =\int_{s}^{t}\int_{s}^{t}R_{i}\binom{\ u\ }{\ v\ }dR_{j}\binom{\ u\ }{\ v\ },
\end{split}
\]
where the integral which appears on the right hand side above can
be viewed as a 2-dimensional (2D) Young's integral (see e.g. Section
6.4 of Friz and Victoir \cite{FV10}). Then we have that 
\[
\begin{split}\int_{s}^{t}\int_{s}^{t}R_{i}\binom{\ u\ }{\ v\ }dR_{j}\binom{\ u\ }{\ v\ } & =\int_{s}^{t}\int_{s}^{t}R_{i}\binom{s,u}{s,v}dR_{j}\binom{\ u\ }{\ v\ }+\int_{s}^{t}\int_{s}^{t}R_{i}\binom{s,u}{s}dR_{j}\binom{\ u\ }{\ v\ }\\
 & \ \ \ +\int_{s}^{t}\int_{s}^{t}R_{i}\binom{s}{s,v}dR_{j}\binom{\ u\ }{\ v\ }+R_{i}\binom{\ s\ }{\ s\ }\int_{s}^{t}\int_{s}^{t}dR_{j}\binom{\ u\ }{\ v\ }\\
 & =:I+II+III+IV.
\end{split}
\]

For the first term $I$, by the Young-Lo\`{e}ve-Towghi inequality (see e.g.
Theorem 6.18 of \cite{FV10}), we have that
\[
\begin{split}I & \leq C|R_{i}|_{\frac{1}{2H}-var;[s,t]^{2}}|R_{j}|_{\frac{1}{2H}-var;[s,t]^{2}}\\
 & \leq C\max\{|R_{i}|_{\frac{1}{2H}-var;[s,t]^{2}}^{2},|R_{j}|_{\frac{1}{2H}-var;[s,t]^{2}}^{2}\}.
\end{split}
\]
Then, by Corollary \ref{Variation-cor}, we have that
\begin{equation}
I=\int_{s}^{t}\int_{s}^{t}R_{i}\binom{s,u}{s,v}dR_{j}\binom{\ u\ }{\ v\ }\leq C|t-s|^{4H}.\label{Holder-1}
\end{equation}

For the second term $II$, by the Young 1D estimate (see e.g. Theorem
6.8 of \cite{FV10}), we have
\[
\begin{split}II & =\int_{s}^{t}R_{i}\binom{s,u}{s}dR_{j}\binom{u}{s,t}\\
 & \leq C\left|R_{i}\binom{\ \cdot\ }{\ s\ }\right|_{\frac{1}{2H}-var;[s,t]}\left|R_{j}\binom{\cdot}{s,t}\right|_{\frac{1}{2H}-var;[s,t]},
\end{split}
\]
where
\[
\begin{split}\left|R_{i}\binom{\ \cdot\ }{\ s\ }\right|_{\frac{1}{2H}-var;[s,t]}^{\frac{1}{2H}} & =\sup_{\mathcal{P}}\sum_{\ell}\left|R_{i}\binom{\ t_{\ell+1}\ }{\ s\ }-R_{i}\binom{\ t_{\ell}\ }{\ s\ }\right|^{\frac{1}{2H}}\\
 & =\sup_{\mathcal{P}}\sum_{\ell}\left|r_{i}(t_{\ell+1}-s)-r_{i}(t_{\ell}-s)\right|^{\frac{1}{2H}}\\
 & \leq\sup_{\mathcal{P}}\sum_{\ell}C_{H}\left|t_{\ell+1}-t_{\ell}\right|\leq C_{H}|t-s|,
\end{split}
\]
and
\[
\begin{split}\left|R_{j}\binom{\cdot}{s,t}\right|_{\frac{1}{2H}-var;[s,t]}^{\frac{1}{2H}} & =\sup_{\mathcal{P}}\sum_{\ell}\left|R_{j}\binom{\ t_{\ell+1}\ }{s,t}-R_{j}\binom{\ t_{\ell}\ }{s,t}\right|^{\frac{1}{2H}}\\
 & =\sup_{\mathcal{P}}\sum_{\ell}\left|\mathbb{E}(X_{t_{\ell},t_{\ell+1}}^{j}X_{s,t}^{j})\right|^{\frac{1}{2H}}\leq\left|R_{j}\right|_{\frac{1}{2H}-var;[s,t]^{2}}^{\frac{1}{2H}}.
\end{split}
\]
In above estimate, function $r_{i}$ is the covariance $r_{i}(t)=\mathbb{E}(X_{s}^{i}X_{s+t}^{i})$.
Thus, we have that 
\begin{equation}
II=\int_{s}^{t}\int_{s}^{t}R_{i}\binom{s,u}{s}dR_{j}\binom{\ u\ }{\ v\ }\leq C|t-s|^{4H}.\label{Holder-2}
\end{equation}

The third term $III$ is the same as the second term $II$
line by line, so
\begin{equation}
III=\int_{s}^{t}\int_{s}^{t}R_{i}\binom{s}{s,v}dR_{j}\binom{\ u\ }{\ v\ }\leq C|t-s|^{4H}.\label{Holder-3}
\end{equation}

For the last term $IV$,
\begin{equation}
\begin{split}IV & =R_{i}\binom{\ s\ }{\ s\ }\left(R_{j}\binom{\ t\ }{\ t\ }-R_{j}\binom{\ s\ }{\ t\ }-R_{j}\binom{\ t\ }{\ s\ }+R_{j}\binom{\ s\ }{\ s\ }\right)\\
 & =r_{i}(0)(2r_{j}(0)-2r_{j}(t-s))\leq C|t-s|^{2H}.
\end{split}
\label{Holder-4}
\end{equation}

Now, combining inequalities (\ref{Holder-1}), (\ref{Holder-2}), (\ref{Holder-3})
and (\ref{Holder-4}), we get that
\begin{equation}
\mathbb{E}\left(\left|\int_{s}^{t}X_{u}^{i}\circ d_{\mathfrak{R}_{1}}\mathbf{X}^{j}\right|^{2}\right)\leq C|t-s|^{4H}+C|t-s|^{2H}.\label{Holder-5}
\end{equation}
Let $s<t$ and that $s,t\in[n-1,n]$, we have that 
\[
\mathbb{E}\left(\left|\int_{s}^{t}X_{u}^{i}\circ d_{\mathfrak{R}_{1}}\mathbf{X}^{j}\right|^{2}\right)\leq C|t-s|^{2H}.
\]

Now we turn to prove the estimate, for arbitrary $p\geq2$, by the
hypercontractivity of Wiener chaos (see Lemma \ref{hypercontract}),
we further have that 
\[
\begin{split}\mathbb{E}[|A_{ij}(t)-A_{ij}(s)|^{p}] & =\mathbb{E}\left(\left|\int_{s}^{t}X_{u}^{i}\circ d_{\mathfrak{R}_{1}}\mathbf{X}^{j}\right|^{p}\right)\\
 & \leq3^{p}(p-1)^{p}\left(\mathbb{E}\left|\int_{s}^{t}X_{u}^{i}\circ d_{\mathfrak{R}_{1}}\mathbf{X}^{j}\right|^{2}\right)^{\frac{p}{2}}\\
 & \leq C|t-s|^{pH}.
\end{split}
\]
Taking that $\Psi(x)=x^{p}$ and $p(x)=x^{H}$, the above inequality implies
that
\[
\mathbb{E}\left(\int_{n}^{n+1}\int_{n}^{n+1}\Psi\left(\frac{|A_{ij}(t)-A_{ij}(s)|}{p(|t-s|)}\right)dsdt\right)\leq C.
\]
Define that 
\[
B_{n}:=\int_{n}^{n+1}\int_{n}^{n+1}\Psi\left(\frac{|A_{ij}(t)-A_{ij}(s)|}{p(|t-s|)}\right)dsdt=\int_{n}^{n+1}\int_{n}^{n+1}\frac{|A_{ij}(t)-A_{ij}(s)|^{p}}{|t-s|^{pH}}dsdt.
\]
Then, for any $q>1$, we get that 
\[
\mathbb{E}\left(\sum_{n=1}^{\infty}\frac{B_{n}}{n^{q}}\right)=\sum_{n=1}^{\infty}\frac{\mathbb{E}(B_{n})}{n^{q}}\leq\sum_{n=1}^{\infty}\frac{C}{n^{q}}<\infty.
\]
Thus there exists an almost surely finite random variable $R$ independent
of $n$ such that
\[
\sum_{n=1}^{\infty}\frac{B_{n}}{n^{q}}\leq R,\ a.s.,
\]
we have that 
\begin{equation}
B_{n}\leq Rn^{q},\ a.s.\ \ \ \forall n\geq1,\ q>1.
\end{equation}
Appling the Garsia-Rodemich-Rumsey inequality, for any $n-1\leq s<t\leq n$,
we get
\[
\begin{split}\left|A_{ij}(t)-A_{ij}(s)\right| & \leq8\int_{0}^{|t-s|}\Psi^{-1}\left(\frac{4B_{n}}{u^{2}}\right)p(du)\leq8H\int_{0}^{1}\left(\frac{4B_{n}}{u^{2}}\right)^{\frac{1}{p}}u^{H-1}du\\
 & =\frac{8H}{H-2/p}(4B_{n})^{\frac{1}{p}}\leq CR^{\frac{1}{p}}n^{\frac{q}{p}},\ \ \ a.s.
\end{split}
\]
for any $p>\frac{2}{H}$ and $1<q<p$. Thus we have completed the proof.
\end{proof}

\subsection{Long time asymptotic of L\'evy area}

\label{longtime-subsec}

For this subsection, we consider the multi-dimensional fOU process, 
which is the solution to the stochastic differential equation
\begin{equation}
dX_{t}=-\Gamma X_{t}dt+\sigma dB_{t}^{H},\ X_{0}=0,\label{fOU-h-2}
\end{equation}
where $\Gamma$ is a symmetric, positive-definite matrix, $\sigma$
is a constant, and $B^{H}=(B^{H,1},B^{H,2},\cdots,B^{H,d})$ is a
$d$-dimensional fBM. Our aim in this section is to show a long time
asymptotic property of L\'evy area $A(t)=\int_{0}^{t}X_{s}\circ d_{\mathfrak{R}_{1}}\mathbf{X}$
of fOU processes $X$ to show that 
\[
\frac{1}{t}A(t)=\frac{1}{t}\int_{0}^{t}X_{s}\circ d_{\mathfrak{R}_{1}}\mathbf{X}\to0,\ a.s.
\]
as $t$ goes to infinity.

The components of the solution for the process $X$ are not independent since
they interact between each other. We first make an orthogonal transformation
for this dynamical system. Since the drift matrix $\Gamma$ is symmetric
and positive-definite, there exists an orthogonal matrix $\overline{\Sigma}$
such that
\begin{equation}
\overline{\Sigma}\Gamma\overline{\Sigma}^{T}=\Lambda,
\end{equation}
where $\Lambda=\mathrm{diag}\{\lambda_{1},\cdots,\lambda_{d}\}$ and
$0<\lambda_{1}\leq\cdots\leq\lambda_{d}$.

Define that $\widetilde{X}_{t}:=\overline{\Sigma}X_{t}$, and $\widetilde{B}_{t}^{H}:=\overline{\Sigma}B_{t}^{H}$.
Since $\overline{\Sigma}$ is an orthogonal matrix, $\widetilde{B}_{t}^{H}$
is still a $d$-dimensional fBM with the Hurst parameter $H$. The stochastic
differential equation (\ref{fOU-h-2}) becomes
\begin{equation}
d\widetilde{X}_{t}=-\Lambda\widetilde{X}_{t}dt+\sigma d\widetilde{B}_{t}^{H}.
\end{equation}
Now the fOU process $\widetilde{X}_{t}$ has independent components, so that
\[
\int_{0}^{t}X_{s}\circ d_{\mathfrak{R}_{1}}\mathbf{X}=\overline{\Sigma}^{T}\left(\int_{0}^{t}\widetilde{X}_{s}\circ d_{\mathfrak{R}_{1}}\mathbf{\widetilde{X}}\right)\overline{\Sigma}.
\]
What we should prove, therefore, is that,
\[
\frac{1}{t}\int_{0}^{t}\widetilde{X}_{s}\circ d_{\mathfrak{R}_{1}}\mathbf{\widetilde{X}}\to0,\ a.s.
\]
as $t$ goes to infinity.

We may ignore the symbol tilde and use $X,B^{H}$ to denote $\widetilde{X}$
and $\widetilde{B}^{H}$, respectively, for simplicity. Now the $d$-dimensional
fOU process $X=(X^{1},X^{2},\cdots,X^{d})$ has independent components
and satisfies that 
\begin{equation}
X_{t}^{i}=\sigma\int_{0}^{t}e^{-\lambda_{i}(t-s)}dB_{s}^{H,i},\ i=1,2,\cdots,d.\label{fOU-indept}
\end{equation}


\subsubsection{On-diagonal case}

\begin{lem}\label{lem-ondiagonal} For the on-diagonal components
of L\'evy area $A(t)=\int_{0}^{t}X_{s}\circ d_{\mathfrak{R}_{1}}\mathbf{X}$,
we have that 
\begin{equation}
\frac{1}{t}A_{ii}(t)=\frac{1}{t}\int_{0}^{t}X_{s}^{i}\circ d_{\mathfrak{R}_{1}}\mathbf{X}^{i}\to0,\ a.s.,\ \forall\ i=1,2,\cdots,d.
\end{equation}
as $t$ tends to infinity.
\end{lem}
\begin{proof} Recall that \begin{equation}
\overline{X}_{t}^{i}=\sigma\int_{-\infty}^{t}e^{-\lambda_{i}(t-s)}dB_{s}^{H,i},\ i=1,2,\cdots,d.\label{fOU-indept-stationary}
\end{equation}
As we have shown, in Proposition \ref{Holder-F(X)-lem}, for any $\alpha>0$,
\begin{equation}
\lim_{t\to\infty}\frac{\overline{X}_{t}^{i}}{t^{\alpha}}=0,\ a.s..\label{t-alpha}
\end{equation}
Since $X_{t}^{i}=\overline{X}_{t}^{i}-e^{-\lambda_{i}t}\overline{X}_{0}^{i}$
and
\[
\int_{0}^{t}X_{s}^{i}\circ d_{\mathfrak{R}_{1}}\mathbf{X}^{i}=\frac{1}{2}(X_{t}^{i})^{2},
\]
from (\ref{t-alpha}) it follows that
\[
\lim_{t\to\infty}\frac{1}{t}\int_{0}^{t}X_{s}^{i}\circ d_{\mathfrak{R}_{1}}\mathbf{X}^{i}=0,\ a.s..
\]
Thus, we have concluded this lemma for the on-diagonal case. \end{proof}

\subsubsection{Off-diagonal case}

Let $X=(X^{1},X^{2},\cdots,X^{d})$ be the $d$-dimensional, stationary
Gaussian process given by (\ref{fOU-indept-stationary}). Its covariance
function is given by
\[
R_{X}(s,t)=\mathrm{diag}(R_{1}(s,t),\cdots,R_{d}(s,t)),
\]
where $R_{i}(s,t):=\mathbb{E}(X_{s}^{i}X_{t}^{i})$.

When $i\neq j$, we have (as proof of equation (\ref{Holder-5}) in
Proposition \ref{Holder-prop}) that
\begin{equation}
\mathbb{E}\left(\left|\int_{0}^{t}X_{s}^{i}\circ d_{\mathfrak{R}_{1}}\mathbf{X}^{j}\right|^{2}\right)\leq Ct^{4H}+Ct^{2H}.
\end{equation}
When $t\geq1$, we have
\begin{equation}
\mathbb{E}\left(\left|\int_{0}^{t}X_{s}^{i}\circ d_{\mathfrak{R}_{1}}\mathbf{X}^{j}\right|^{2}\right)\leq Ct^{4H}.\label{4H-ineq}
\end{equation}

Now we define $A_{ij}(t)=\int_{0}^{t}X_{s}^{i}\circ d_{\mathfrak{R}_{1}}\mathbf{X}^{j}$,
as in subsection \ref{holder-levy-subsec}, and that $Z_{n}^{ij}:=n^{-2H}A_{ij}(n)$. 
We first show that, when $t=n\in\mathbb{N}$ (discrete sequence), we
have that 
\begin{equation}
\lim_{n\to\infty}\frac{1}{n}A_{ij}(n)=0,\ a.s..
\end{equation}

\begin{prop}\label{discrete-converge} For the discrete sequence
$\{\frac{1}{n}A_{ij}(n),n\geq1\}$ and $H\in(\frac{1}{3},\frac{1}{2})$,
we have that 
\begin{equation}
\frac{1}{n}A_{ij}(n)\to0,\ a.s.
\end{equation}
as $n$ goes to infinity. \end{prop} \begin{proof} By the inequality
(\ref{4H-ineq}), we have
\[
\mathbb{E}|A_{ij}(n)|^{2}\leq Cn^{4H}.
\]
Then
\[
\sup_{n}\mathbb{E}|Z_{n}^{ij}|^{2}\leq C.
\]
According Proposition 15.20 of \cite{FV10}, we know that $Z_{n}^{ij}$
belongs to the second Wiener chaos $\mathcal{C}^{2}(\mathbb{P})$.
By Lemma \ref{hypercontract}, we have that 
\[
\sup_{n}\mathbb{E}|Z_{n}^{ij}|^{p}\leq3^{p}(p-1)^{p}\sup_{n}(\mathbb{E}|Z_{n}^{ij}|^{2})^{\frac{p}{2}}<\infty.
\]
For any $\epsilon>0$, by the Chebyshev inequality, we have that 
\[
\mathbb{P}\left(|A_{ij}(n)|>n\epsilon\right)=\mathbb{P}\left(|Z_{n}^{ij}|>n^{1-2H}\epsilon\right)\leq\frac{1}{n^{p(1-2H)}\epsilon^{p}}\sup_{n}\mathbb{E}|Z_{n}^{ij}|^{p},
\]
where $p>\frac{1}{1-2H}$.

Then,
\[
\sum_{n}\mathbb{P}\left(|A_{ij}(n)|>n\epsilon\right)\leq\sum_{n}\frac{C}{n^{p(1-2H)}\epsilon^{p}}<\infty.
\]
The almost sure convergence follows from the Borel-Cantelli lemma.
\end{proof}

Now we can conclude this subsection and show that the limit for
arbitrary $t$ rather than at discrete time $\mathbb{N}_{+}$.

\begin{thm}\label{limit-Strat-xdx} Suppose stochastic process $X_{t}$
is the fOU process which is the solution to stochastic differential equation
(\ref{fOU-h-2}) and $\Gamma$ is symmetric and positive-definite.
Then
\[
\frac{1}{t}A(t)=\frac{1}{t}\int_{0}^{t}X_{s}\otimes\circ d_{\mathfrak{R}_{1}}\mathbf{X}\to0,\ a.s.,
\]
as $t\to\infty$, where the above integral is in the Stratonovich sense.
\end{thm}

\begin{proof} First, assume that $X$ is the stationary fOU process
as in equation (\ref{fOU-indept-stationary}). The on-diagonal case
is prove in Lemma \ref{lem-ondiagonal}. For the off-diagonal case,
since
\begin{equation}
\frac{1}{t}|A_{ij}(t)|\leq\frac{1}{t}|A_{ij}(t)-A_{ij}(n)|+\frac{n}{t}\frac{1}{n}|A_{ij}(n)|
\end{equation}
and setting that $n=[t]$, by Proposition \ref{Holder-prop}, we have that
the first term on the right hand side is controlled by $\widetilde{C}t^{-1}n^{\beta}\leq\widetilde{C}n^{\beta-1}\to0,\ a.s.$.
The second term also tends to zero by Proposition \ref{discrete-converge}.
Thus we have completed the proof of Theorem \ref{limit-Strat-xdx}
when the fOU process $X$ is stationary.

If $X$ is not stationary but starts at point $0$ at $t=0$,
we can also prove this asymptotic for the Stratonovich integrals.
Now let $\overline{X}=(\overline{X}^{1},\cdots,\overline{X}^{d})$
be the stationary version as above. Then the fOU process is  $X_{t}^{i}=\overline{X}_{t}^{i}-e^{-\lambda_{i}t}\overline{X}_{0}^{i}$,
$i=1,2,\cdots,d$, so
\[
\begin{split}\frac{1}{t}\int_{0}^{t}X_{u}^{i}\circ d_{\mathfrak{R}_{1}}\mathbf{X}^{j} & =\frac{1}{t}\int_{0}^{t}\overline{X}_{u}^{i}\circ d_{\mathfrak{R}_{1}}\overline{\mathbf{X}}^{j}+\frac{1}{t}\int_{0}^{t}\lambda_{j}e^{-\lambda_{j}u}\overline{X}_{u}^{i}du\overline{X}_{0}^{j}\\
 & \ \ \ -\frac{1}{t}\int_{0}^{t}e^{-\lambda_{i}u}d\overline{X}_{u}^{j}\overline{X}_{0}^{i}-\frac{1}{t}\int_{0}^{t}\lambda_{j}e^{-(\lambda_{i}+\lambda_{j})u}du\overline{X}_{0}^{i}\overline{X}_{0}^{j},
\end{split}
\]
where the last three integrals are Young's integrals.

The first term on the right hand side tends to zero almost surely,
which has been prove above. The last term also goes to zero almost
surely, which can be prove easily. For the second and third terms,
we can see that $\int_{0}^{t}\lambda_{j}e^{-\lambda_{j}u}\overline{X}_{u}^{i}du$
and $\int_{0}^{t}e^{-\lambda_{i}u}d\overline{X}_{u}^{j}$ are two
Gaussian processes. By almost the same arguments to the proof of the
limit $\frac{1}{t}\int_{0}^{t}\overline{X}_{u}^{i}\circ d_{\mathfrak{R}_{1}}\overline{\mathbf{X}}^{j}\to0,\ a.s$,
we can also prove that the second and third terms both converge to
zero almost surely. Here we just give a sketch of proof for the second
term.

Define that $Z_{t}=\int_{0}^{t}\lambda_{j}e^{-\lambda_{j}u}\overline{X}_{u}^{i}du$,
and $\xi=\overline{X}_{0}^{j}$. First, we show that $\frac{1}{n}(\xi Z_{n})\to0,\ a.s$
for the integer subsequence. Since
\[
\begin{split}\mathbb{E}|Z_{n}|^{2} & =\mathbb{E}\left(\int_{0}^{n}\lambda_{j}e^{-\lambda_{j}u}\overline{X}_{u}^{i}du\right)^{2}=\int_{0}^{n}\int_{0}^{n}r_{i}(u-v)e^{-\lambda_{j}(u+v)}dudv\\
 & \leq\max_{t\geq0}|r_{i}(t)|\int_{0}^{n}\int_{0}^{n}e^{-\lambda_{j}(u+v)}dudv\leq\frac{C}{\lambda_{j}^{2}}(e^{-\lambda_{j}n}-1)^{2}\leq\widetilde{C},
\end{split}
\]
where $C,\widetilde{C}$ are independent of $n$, 
\[
\mathbb{P}\left(\frac{1}{n}|\xi Z_{n}|>\varepsilon\right)\leq\frac{\mathbb{E}|\xi Z_{n}|^{2}}{n^{2}\varepsilon^{2}}\leq\frac{(\mathbb{E}\xi^{4})^{\frac{1}{2}}+(\mathbb{E}Z_{n}^{4})^{\frac{1}{2}}}{n^{2}\varepsilon^{2}}\leq\frac{C}{n^{2}\varepsilon^{2}},
\]
by the Borel-Cantelli lemma, we have prove that $\frac{1}{n}(\xi Z_{n})\to0,\ a.s$.

Now we show, for any $n\geq1$ and any $s,t\in[n,n+1]$, that there exist
a constant $\beta\in(0,1)$ and an almost surely finite random variable
$R$ such that $|Z_{t}-Z_{s}|\leq Rn^{\beta},\ a.s.$. Since
\[
\mathbb{E}|Z_{t}-Z_{s}|^{2}=\mathbb{E}\left(\int_{s}^{t}\lambda_{j}e^{-\lambda_{j}u}\overline{X}_{u}^{i}du\right)^{2}=\int_{s}^{t}\int_{s}^{t}r_{i}(u-v)e^{-\lambda_{j}(u+v)}dudv\leq C|t-s|^{2},
\]
where $C$ is a universal constant, applying the Garsia-Rodemich-Rumsey
inequality as Proposition \ref{Holder-prop}, we get that $|Z_{t}-Z_{s}|\leq Rn^{\beta},\ a.s.$
Then, choosing that $n=[t]$,
\[
\frac{1}{t}|\xi Z_{t}|\leq\frac{1}{t}|\xi||Z_{t}-Z_{n}|+\frac{n}{t}\frac{1}{n}|\xi Z_{n}|\leq R|\xi|n^{\beta-1}+\frac{1}{n}|\xi Z_{n}|\to0,\ a.s..
\]
Thus, we prove the limit of the second term. The third term follows as above. By taking an orthogonal transformation for $X$ (independent
components), we get the same limit for Stratonovich integral of solution
to Equation (\ref{fOU-h-2}). This concludes the theorem.
\end{proof}

\section{Pathwise stable estimators}
\label{sec-CRPE}

\subsection{Continuous rough path estimator}

In this section, let $X$ be the fOU process, i.e. the solution
to the following stochastic differential equation:
\begin{equation}
dX_{t}=-\Gamma X_{t}dt+\Sigma dB_{t}^{H}.\label{fOU-h-3}
\end{equation}
We construct an estimator based on continuous observation via
rough path theory. We suppose that the rough path enhancement $(X_{0,t}(\omega),\mathbb{X}_{0,t}(\omega))$
of fOU process $X_{t}(\omega)$ could be continuously observed in
the It\^o sense defined in Section \ref{sec-RP}. This may leave users
with the question of how to understand data as a rough path in practice, but there are in fact works on how to apply inverse data
to rough paths. We recommend those who may be interested in these
questions to look at the literature on rough path analysis, in particular, \cite{BD15}.

For the construction of the estimator, we adapt the idea of the least square
estimator of Hu and Nualart \cite{HN10}. Hu and Nualart have derived the estimator
in the one dimensional case, which is formally taken as the minimizer
\begin{equation}
\widehat{\gamma}_{T}:=\arg\min_{\gamma\in\Theta}\int_{0}^{T}|\dot{X}_{t}-(-\gamma X_{t})|^{2}dt,
\end{equation}
where $\Theta$ is the parameter space. In multi-dimensional case,
we consider (formally) the estimator as the minimizer, which is also discussed in \cite{HNZ19}, so that 
\begin{equation}
\widehat{\Gamma}_{t}:=\arg\min_{\Gamma\in\Theta}\int_{0}^{t}\|\Sigma^{-1}\dot{X}_{s}-(-\Gamma\Sigma^{-1}X_{s})\|^{2}ds,\ t\in[0,T],
\end{equation}
which leads to the solution
\begin{equation}\label{continuous-Gamma}
\widehat{\Gamma}_{t}=-\mathcal{L}_{t}^{-1}S_{t},
\end{equation}
where
\begin{align}
\mathcal{L}_{t} & =\int_{0}^{t}(I\otimes X_{s})^{T}Q^{-1}(I\otimes X_{s})ds\in L(V,V^{*}),\\
S_{t} & =\int_{0}^{t}(I\otimes X_{s})^{T}Q^{-1}d_{\mathfrak{R}_{1}}\mathbf{X}\in V^{*},
\end{align}
and space $V=\mathbb{R}^{d\times d}$, $\mathcal{L}_{t}^{-1}$ is
the inverse of $\mathcal{L}_{t}$, $Q=\Sigma\Sigma^{T}$, $I\otimes X=(\delta_{j}^{i}X^{k})_{i,j,k=1,\cdots,d}$,
and $M^{T}$ denotes transpose of matrix $M$. The integral $S_{t}$
is taken as the It\^o rough integral of $X$ defined as in Section \ref{sec-RP}.
We call this estimator a \textit{rough path estimator}.

When $\Sigma=\sigma I$ ($I$ is identity matrix, $\sigma$ is a constant),
the estimator becomes
\begin{equation}
\widehat{\Gamma}_{t}^{T}=-\left(\int_{0}^{t}X_{s}\otimes X_{s}ds\right)^{-1}\left(\int_{0}^{t}X_{s}\otimes d_{\mathfrak{R}_{1}}\mathbf{X}\right).\label{estimator}
\end{equation}
Actually, we can make a rotation to dynamical system (\ref{fOU-h-3}),
i.e. act $\Sigma^{-1}$ to $X_{t}$, then we get the above diagonal
case. Thus, without loss of generality, we can suppose that $\Sigma=\sigma I$.

Now we give two examples for cases $d=1,2$. For the one dimensional case,
the rough path estimator is
\begin{equation}
\widehat{\gamma}_{t}=-\frac{\int_{0}^{t}X_{s}d_{\mathfrak{R}_{1}}\mathbf{X}}{\int_{0}^{t}X_{s}^{2}ds}=-\frac{\mathbb{X}_{0,t}+X_{0}X_{0,t}}{\int_{0}^{t}X_{s}^{2}ds}.
\end{equation}
For $d=2$, the transpose of the rough path estimator is
\begin{equation}
\begin{split}\widehat{\Gamma}_{t}^{T} & =-\frac{1}{\det(\mathcal{L}_{t}(X))}\begin{pmatrix}\int_{0}^{t}(X_{s}^{2})^{2}ds & -\int_{0}^{t}X_{s}^{1}X_{s}^{2}ds\\
-\int_{0}^{t}X_{s}^{1}X_{s}^{2}ds & \int_{0}^{t}(X_{s}^{1})^{2}ds
\end{pmatrix}\\
 & \ \ \ \ \ \ \ \ \ \ \ \ \ \ \ \ \ \ \ \ \ \ \ \ \ \ \ \times\begin{pmatrix}\int_{0}^{t}X_{s}^{1}d_{\mathfrak{R}_{1}}\mathbf{X}^{1} & \int_{0}^{t}X_{s}^{1}d_{\mathfrak{R}_{1}}\mathbf{X}^{2}\\
\int_{0}^{t}X_{s}^{2}d_{\mathfrak{R}_{1}}\mathbf{X}^{1} & \int_{0}^{t}X_{s}^{2}d_{\mathfrak{R}_{1}}\mathbf{X}^{2}
\end{pmatrix},
\end{split}
\end{equation}
where
\begin{align}
\det(\mathcal{L}_{t}(X)) & =\int_{0}^{t}(X_{s}^{1})^{2}ds\int_{0}^{t}(X_{s}^{2})^{2}ds-\left(\int_{0}^{t}X_{s}^{1}X_{s}^{2}ds\right)^{2},\\
\int_{0}^{t}X_{s}^{i}d_{\mathfrak{R}_{1}}\mathbf{X}^{j} & =\mathbb{X}_{0,t}^{ij}+X_{0}^{i}X_{0,t}^{j},\ \ i,j=1,2.
\end{align}
As a remark, we mention that here $X(\omega),\mathbb{X}(\omega)$
and $\widehat{\Gamma}(\omega)$ are pathwise-defined, almost surely.

\subsection{Strong consistency}

Now we consider the asymptotic behavior of the rough path estimator
$\widehat{\Gamma}_{t}$. The solution $X$ to (\ref{fOU-h-3}) is
given by
\begin{equation}
X_{t}=e^{-\Gamma t}X_{0}+\int_{0}^{t}e^{-\Gamma(t-s)}\Sigma dB_{s}^{H}.
\end{equation}
Without loss of generality, we suppose that $X_{0}=0$.

In what follows, we will prove chain rules for our rough integrals,
and then show the almost sure convergence of our rough path estimator.

\subsubsection{Chain rules}

First, we have the following lemma: \begin{lem}\label{chain-lem}
For $H\in(\frac{1}{3},\frac{1}{2}]$, we have that 
\begin{equation}
\int_{0}^{t}X_{s}\otimes d_{\mathfrak{R}_{1}}\mathbf{X}=-\left(\int_{0}^{t}X_{s}\otimes X_{s}ds\right)\Gamma^{T}+\sigma\int_{0}^{t}X_{s}\otimes d_{\mathfrak{R}_{1}}\mathbf{B}^{H}.
\end{equation}
Here, the integrals can be either Stratonovich's or It\^o's rough integrals.
\end{lem} \begin{proof} We use the relationship between almost rough
paths and rough paths, see Theorem 3.2.1 in \cite{LQ02}, to prove
this lemma. To simplify notations, we show the $d=1$ case, i.e. prove that 
\begin{equation}
\int_{0}^{t}X_{s}d_{\mathfrak{R}_{1}}\mathbf{X}=-\gamma\int_{0}^{t}(X_{s})^{2}ds+\sigma\int_{0}^{t}X_{s}d_{\mathfrak{R}_{1}}\mathbf{B}^{H}.\label{chain}
\end{equation}
First, by the theory of rough differential equations and (\ref{RDE}),
we know that
\begin{align}
Z_{s,t} & \simeq f(Z_{s})Z_{s,t}+Df(Z_{s})\mathbb{Z}_{s,t},\label{Z-1}\\
\mathbb{Z}_{s,t} & \simeq f(Z_{s})\otimes f(Z_{s})\mathbb{Z}_{s,t},\label{Z-2}
\end{align}
where the right hand sides are actually almost rough paths associated
$\mathbf{Z}_{s,t}$, and $\simeq$ means the difference is controlled
by $\omega(s,t)^{\theta}$ with $\theta>1$, for all $(s,t)\in\Delta$.
By Equation (\ref{Z-1}), we have that 
\[
Z_{s,t}\simeq\left(B_{s,t}^{H},-\gamma X_{s}(t-s)+\sigma B_{s,t}^{H},t-s\right)+\left(0,-\gamma\int_{s}^{t}X_{s,u}du,0\right).
\]
Since
\[
\left|\int_{s}^{t}X_{s,u}du\right|=\left|\int_{s}^{t}X_{u}du-X_{s}(t-s)\right|=o(|t-s|),
\]
we have that 
\[
Z_{s,t}\simeq\left(B_{s,t}^{H},-\gamma X_{s}(t-s)+\sigma B_{s,t}^{H},t-s\right).
\]
This implies that 
\begin{equation}
X_{s,t}\simeq-\gamma X_{s}(t-s)+\sigma B_{s,t}^{H}.\label{X-1}
\end{equation}
Actually, the above formula could be found from the stochastic differential
equation (\ref{fOU-h-3}) directly. Now by (\ref{Z-2}), we have that
\[
\mathbb{Z}_{s,t}\simeq\begin{pmatrix}\mathbb{B}_{s,t}^{H} & \int_{s}^{t}B_{s,u}^{H}dX_{u} & \int_{s}^{t}B_{s,u}^{H}du\\
M_{1} & M_{2} & M_{3}\\
\int_{s}^{t}(u-s)dB_{u}^{H} & \int_{s}^{t}(u-s)dX_{u} & \frac{1}{2}(t-s)^{2}
\end{pmatrix},
\]
where
\begin{align*}
M_{1} & =\sigma\mathbb{B}_{s,t}^{H}-\gamma X_{s}\int_{s}^{t}(u-s)dB_{u}^{H},\\
M_{2} & =\sigma\int_{s}^{t}B_{s,u}^{H}dX_{u}-\gamma X_{s}\int_{s}^{t}(u-s)dX_{u},\\
M_{3} & =\sigma\int_{s}^{t}B_{s,u}^{H}du-\frac{1}{2}\gamma X_{s}(t-s)^{2}.
\end{align*}
Hence,
\begin{align}
 & \int_{s}^{t}X_{s,u}dB_{u}^{H}\simeq\sigma\mathbb{B}_{s,t}^{H}-\gamma X_{s}\int_{s}^{t}(u-s)dB_{u}^{H}\simeq\sigma\mathbb{B}_{s,t}^{H},\label{XB}\\
 & \ \ \mathbb{X}_{s,t}\simeq\sigma\int_{s}^{t}B_{s,u}^{H}dX_{u}-\gamma X_{s}\int_{s}^{t}(u-s)dX_{u}\simeq\sigma\int_{s}^{t}B_{s,u}^{H}dX_{u},\label{XX}\\
 & \int_{s}^{t}X_{s,u}du\simeq\sigma\int_{s}^{t}B_{s,u}^{H}du-\frac{1}{2}\gamma X_{s}(t-s)^{2}=o(|t-s|).
\end{align}
Combining (\ref{X-1}) and (\ref{XX}), we further have $\mathbb{X}_{s,t}\simeq\sigma^{2}\mathbb{B}_{s,t}^{H}$.

Now using the results above, we can show Equation (\ref{chain}), 
since we have that 
\[
\begin{split}LHS & \simeq X_{s}X_{s,t}+\mathbb{X}_{s,t}\\
 & \simeq-\gamma(X_{s})^{2}(t-s)+\sigma X_{s}B_{s,t}^{H}+\sigma^{2}\mathbb{B}_{s,t}^{H}\\
 & \simeq RHS.
\end{split}
\]
Thus we have completed the proof of the lemma. \end{proof} As a
corollary, we have \begin{cor}\label{cor1} For $H\in(\frac{1}{3},\frac{1}{2}]$,
the rough path estimator $\widehat{\Gamma}_{t}$ has the following
expression:
\begin{equation}
\widehat{\Gamma}_{t}^{T}=\Gamma^{T}-\left(\int_{0}^{t}X_{s}\otimes X_{s}ds\right)^{-1}\left(\int_{0}^{t}X_{s}\otimes d_{\mathfrak{R}_{1}}\mathbf{B}^{H,\text{It\^o}}\right).
\end{equation}
\end{cor}

\subsubsection{Almost sure convergence}

In order to establish the strong consistency of the rough path estimator
$\widehat{\Gamma}_{t}$, i.e.,
\begin{equation}
\widehat{\Gamma}_{t}\to\Gamma,\ a.s.\ \text{as}\ t\to\infty,\label{strongcns}
\end{equation}
our aim now is to prove that
\begin{equation}
\frac{1}{t}\int_{0}^{t}X_{s}\otimes X_{s}ds\to C_{1}(H),\ a.s.,
\end{equation}
\begin{equation}
\frac{1}{t}\int_{0}^{t}X_{s}\otimes d_{\mathfrak{R}_{1}}\mathbf{B}^{H,\text{It\^o}}\to0,\ a.s..
\end{equation}
Then, combining the Slutsky Theorem and Corollary \ref{cor1}, we can get (\ref{strongcns}).

\begin{prop}\label{limit-x2} Suppose that the stochastic process $X_{t}$
is the fOU process to stochastic differential equation (\ref{fOU-h-3})
and $\Gamma$ is symmetric and positive-definite. Then
\[
\frac{1}{t}\int_{0}^{t}X_{s}\otimes X_{s}ds\to C_{1}(H),\ a.s.,
\]
where the above integral on the left hand side is the Lebesgue integral
and the constant matrix is $C_{1}(H)=\sigma^{2}H\int_{0}^{\infty}x^{2H-1}e^{-\Gamma x}dx$.
\end{prop} \begin{proof} Define the process
\begin{equation}
\overline{X}_{t}:=\sigma\int_{-\infty}^{t}e^{-\Gamma(t-s)}dB_{s}^{H}.\label{fOU-ergd-high}
\end{equation}
Then the process $\overline{X}$ is stationary Gaussian process and
is ergodic (see Section \ref{longtime-subsec}). By the ergodic
theorem (see \cite{CKM03}), we have that 
\begin{equation}
\lim_{t\to\infty}\frac{1}{t}\int_{0}^{t}\overline{X}_{s}\otimes\overline{X}_{s}ds=\mathbb{E}(\overline{X}_{0}\otimes\overline{X}_{0}),\ a.s..
\end{equation}
Since
\[
X_{t}=\overline{X}_{t}-e^{-\Gamma t}\overline{X}_{0},
\]
so that
\begin{equation}
\lim_{t\to\infty}\frac{1}{t}\int_{0}^{t}X_{s}\otimes X_{s}ds=\mathbb{E}(\overline{X}_{0}\otimes\overline{X}_{0}),\ a.s.
\end{equation}
For the right hand side, applying integration by parts, we have that 
\[
\begin{split}\mathbb{E}(\overline{X}_{0}\otimes\overline{X}_{0}) & =\sigma^{2}\mathbb{E}\left(\int_{-\infty}^{0}e^{\Gamma s}dB_{s}^{H}\right)\otimes\left(\int_{-\infty}^{0}e^{\Gamma s}dB_{s}^{H}\right)\\
 & =\sigma^{2}\Gamma^{2}\mathbb{E}\left(\int_{-\infty}^{0}\int_{-\infty}^{0}e^{\Gamma(s+u)}(B_{s}^{H}\otimes B_{u}^{H})duds\right)\\
 & =\sigma^{2}\Gamma^{2}\int_{0}^{\infty}\int_{0}^{\infty}e^{-\Gamma(s+u)}\frac{1}{2}I(s^{2H}+u^{2H}-|s-u|^{2H})duds\\
 & =\sigma^{2}\Gamma\int_{0}^{\infty}x^{2H}e^{-\Gamma x}dx,
\end{split}
\]
and
\[
\Gamma\int_{0}^{\infty}x^{2H}e^{-\Gamma x}dx=H\int_{0}^{\infty}x^{2H-1}e^{-\Gamma x}dx.
\]
Thus we have prove this lemma. \end{proof}

\begin{prop}\label{limit-Ito-xdb} Suppose that the stochastic process $X_{t}$
is the fOU process to stochastic differential equation (\ref{fOU-h-3})
and $\Gamma$ is symmetric and positive-definite. Then
\[
\frac{1}{t}\int_{0}^{t}X_{s}\otimes d_{\mathfrak{R}_{1}}\mathbf{B}^{H,\text{It\^o}}\to0,\ a.s.,
\]
where $\mathbf{B}^{H,\text{It\^o}}$ is It\^o type rough path enhancement
of fBM $B_{t}^{H}$ as in Section \ref{sec-RP} with $H\in(\frac{1}{3},\frac{1}{2}]$.
\end{prop} \begin{proof} Applying integration by parts, we have
\begin{equation}
X_{t}=\sigma\int_{0}^{t}e^{-\Gamma(t-s)}dB_{s}^{H}=\sigma\left(B_{t}^{H}-\Gamma\int_{0}^{t}e^{-\Gamma(t-s)}B_{s}^{H}ds\right).
\end{equation}
By the definitions of the It\^o integration and the Stratonovich integration with
respect to fBM for the fOU process (see rough differential equation (\ref{RDE})),
we have that 
\begin{equation}
\int_{0}^{t}X_{s}\otimes d_{\mathfrak{R}_{1}}\mathbf{B}^{H,\text{It\^o}}=\int_{0}^{t}X_{s}\otimes\circ d_{\mathfrak{R}_{1}}\mathbf{B}^{H,\text{Str}}-\sigma\varphi(t),\label{chain-xdb}
\end{equation}
where
\[
\varphi(t)=\frac{1}{2}It^{2H}-U(t),
\]
and
\[
U(t)=H\Gamma\int_{0}^{t}\int_{0}^{s}e^{-\Gamma(s-u)}(s^{2H-1}-(s-u)^{2H-1})duds.
\]
The first term on the right hand side is defined as Stratonovich
integral, and has the following expression (by Lemma \ref{chain-lem})
\begin{equation}
\sigma\int_{0}^{t}X_{s}\otimes\circ d_{\mathfrak{R}_{1}}\mathbf{B}^{H,\text{Str}}=\int_{0}^{t}X_{s}\otimes\circ d_{\mathfrak{R}_{1}}\mathbf{X}+\Gamma\int_{0}^{t}X_{s}\otimes X_{s}ds.\label{chain-strat}
\end{equation}

Now we represent $U(t)$ as
\[
\begin{split}U(t) & =\frac{1}{2}It^{2H}-H\int_{0}^{t}e^{-\Gamma s}s^{2H-1}ds-H\Gamma t\int_{0}^{t}e^{-\Gamma s}s^{2H-1}ds\\
 & \ \ \ \ +H\Gamma\int_{0}^{t}e^{-\Gamma s}s^{2H}ds,
\end{split}
\]
and we have that 
\[
\begin{split}\varphi(t) & =H\int_{0}^{t}e^{-\Gamma s}s^{2H-1}ds+H\Gamma t\int_{0}^{t}e^{-\Gamma s}s^{2H-1}ds\\
 & \ \ \ \ -H\Gamma\int_{0}^{t}e^{-\Gamma s}s^{2H}ds,
\end{split}
\]
Since $\int_{0}^{t}e^{-\Gamma s}s^{\alpha-1}ds\uparrow\int_{0}^{\infty}e^{-\Gamma s}s^{\alpha-1}ds\leq C$
as $t\to\infty$,
\begin{equation}
\lim_{t\to\infty}\frac{1}{t}\varphi(t)=H\Gamma\int_{0}^{\infty}e^{-\Gamma s}s^{2H-1}ds,\ a.s.\label{limit-phi}
\end{equation}

Then Combining Equations (\ref{chain-xdb}), (\ref{chain-strat}), 
Theorem \ref{limit-Strat-xdx}, and Proposition \ref{limit-x2}, for
\[
\sigma\int_{0}^{t}X_{s}\otimes d_{\mathfrak{R}_{1}}\mathbf{B}^{H,\text{It\^o}}=\int_{0}^{t}X_{s}\otimes\circ d_{\mathfrak{R}_{1}}\mathbf{X}+\Gamma\int_{0}^{t}X_{s}\otimes X_{s}ds-\sigma^{2}\varphi(t),
\]
we have that 
\[
\lim_{t\to\infty}\frac{\sigma}{t}\int_{0}^{t}X_{s}\otimes d_{\mathfrak{R}_{1}}\mathbf{B}^{H,\text{It\^o}}=\Gamma C_{1}(H)-\sigma^{2}H\Gamma\int_{0}^{\infty}e^{-\Gamma s}s^{2H-1}ds=0,\ a.s..
\]
This concludes the proposition. \end{proof}

As a corollary of Theorem \ref{limit-Strat-xdx}, now we have the
following statement, in which the integral is in the It\^o sense: \begin{cor}\label{limit-Ito-xdx}
Suppose that the stochastic process $X_{t}$ is the fOU process to stochastic
differential equation (\ref{fOU-h-3}) and $\Gamma$ is symmetric
and positive-definite. Then
\begin{equation}
\frac{1}{t}\int_{0}^{t}X_{s}\otimes d_{\mathfrak{R}_{1}}\mathbf{X}\to C_{2}(H),\ a.s.,
\end{equation}
where the above integral is in the It\^o sense. \end{cor} \begin{proof}
As we can see from the definition of the rough integral and Lemma \ref{chain-lem}, 
\[
\begin{split}\int_{s}^{t}X\otimes d_{\mathfrak{R}_{1}}\mathbf{X} & \simeq X_{s}X_{s,t}+\mathbb{X}_{s,t}\\
 & \simeq X_{s}X_{s,t}+\sigma^{2}\mathbb{B}_{s,t}^{H,\text{It\^o}}\\
 & \simeq X_{s}X_{s,t}+\sigma^{2}(\mathbb{B}_{s,t}^{H,\text{Str}}-\varphi_{s,t})\\
 & \simeq\int_{s}^{t}X\otimes\circ d_{\mathfrak{R}_{1}}\mathbf{X}-\sigma^{2}\varphi_{s,t}
\end{split}
\]
for any $(s,t)\in\Delta$. Thus we have that 
\begin{equation}
\int_{0}^{t}X_{s}\otimes d_{\mathfrak{R}_{1}}\mathbf{X}=\int_{0}^{t}X_{s}\otimes\circ d_{\mathfrak{R}_{1}}\mathbf{X}-\sigma^{2}\varphi(t).
\end{equation}
By Theorem \ref{limit-Strat-xdx} and the limit in Equation (\ref{limit-phi}),
we get that 
\begin{equation}
\lim_{t\to\infty}\frac{1}{t}\int_{0}^{t}X_{s}\otimes d_{\mathfrak{R}_{1}}\mathbf{X}=-\sigma^{2}H\Gamma\int_{0}^{\infty}e^{-\Gamma s}s^{2H-1}ds=:C_{2}(H),\ a.s.
\end{equation}
In addition, by the definition of $C_{1}(H)$, we also have the relation
between $C_{2}(H)$ and $C_{1}(H)$ as $C_{2}(H)=-\Gamma C_{1}(H)$.
\end{proof}

Now we have strong consistency of the rough path estimator $\widehat{\Gamma}_{t}$
as $t$ tends to infinity. \begin{thm} Suppose that $\Gamma$ is a parametric
matrix and that it is symmetric and positive-definite. Let $\widehat{\Gamma}_{t}$
be the rough path estimator as (\ref{estimator}) of $\Gamma$ for
the stochastic differential equation (\ref{fOU-h-3}). Then
\begin{equation}
\widehat{\Gamma}_{t}\to\Gamma,\ a.s.,\ \text{as}\ t\to\infty.
\end{equation}
\end{thm} \begin{proof} Applying Proposition \ref{limit-x2} and
Proposition \ref{limit-Ito-xdb}, and by the Slutsky Theorem, we have that 
\[
\left(\frac{1}{t}\int_{0}^{t}X_{s}\otimes X_{s}ds\right)^{-1}\left(\frac{1}{t}\int_{0}^{t}X_{s}\otimes d_{\mathfrak{R}_{1}}\mathbf{B}^{H}\right)\to0,\ a.s.
\]
as $t$ goes to infinity. Hence, by Corollary \ref{cor1}, the rough
path estimator $\widehat{\Gamma}_{t}$ almost surely converges to
$\Gamma$. \end{proof}

\begin{rem} Suppose that we take the stochastic integral $\int_{0}^{t}X_{s}\otimes d_{\mathfrak{R}_{1}}\mathbf{X}$
in the rough path estimator $\widehat{\Gamma}_{t}$ (equation (\ref{estimator}))
as the Stratonovich rough integral rather than the It\^o rough integral as above.
We can see that
\begin{equation}
\widehat{\Gamma}_{t}\to0,\ a.s.,\ \text{as}\ t\to\infty,
\end{equation}
by Theorem \ref{limit-Strat-xdx} and Proposition \ref{limit-x2}.
That is to say, we cannot use the Stratonovich rough integral to do this
estimation problem. \end{rem}

\begin{rem}
The explicit dependence of the L\'evy area of fOU processes on the drift parameter $\Gamma$ can make it challenging to apply the estimator in practical observations. However, as suggested by equation (\ref{continuous-Gamma}), we can treat it as an equation for $\Gamma$ and solve it iteratively.
\end{rem}

\subsection{Pathwise stability}

In this subsection, we will show that our rough path estimator is
pathwise stable and robust. Note that $\widehat{\Gamma}_{T}$ is a
functional on the path space $C([0,T],\mathbb{R}^{d})$, or exactly
on the rough path space $\Omega_{p}([0,T],\mathbb{R}^{d})$. For every
observation sample path $X(\omega)$ or rough path enhancement $\mathbf{X}(\omega)=(X(\omega),\mathbb{X}(\omega))$,
one has a corresponding estimator $\widehat{\Gamma}_{T}(X(\omega))$
or $\widehat{\Gamma}_{T}(\mathbf{X}(\omega))=\widehat{\Gamma}_{T}((X(\omega),\mathbb{X}(\omega)))$.
In what follows, we will use the rough path notation rather than the 
sample path, since our continuous rough path estimator depends on
$\mathbf{X}(\omega)=(X(\omega),\mathbb{X}(\omega))$ rather than just
the first level sample path $X(\omega)$.

A natural question regarding the robustness of the estimator arises: if two observations
$\mathbf{X}$ and $\widetilde{\mathbf{X}}$ are very close in some
sense, e.g. the uniform distance or the $p$-variation distance etc., does
it give rise to close estimations $\widehat{\Gamma}_{T}(\mathbf{X})\thickapprox\widehat{\Gamma}_{T}(\widetilde{\mathbf{X}})$?
In other words, is the estimator $\widehat{\Gamma}_{T}(\cdot)$ continuous
in some distance?

Actually, the rough path idea gives us a good solution to this problem.
As is well-known, in rough path space, rough integration is continuous
with respect to $p$-variation distance. Now we first recall the $p$-variation
rough path distance $d_{p}$:
\begin{equation}
d_{p}(\mathbf{X},\mathbf{Y})=\max_{i=1,2}\sup_{\mathcal{P}}\left(\sum_{\ell}|\mathbf{X}_{t_{\ell-1},t_{\ell}}^{i}-\mathbf{Y}_{t_{\ell-1},t_{\ell}}^{i}|^{\frac{p}{i}}\right)^{\frac{i}{p}}.
\end{equation}
Here $\mathbf{X}=(\mathbf{X}^{1},\mathbf{X}^{2})$ and $\mathbf{Y}=(\mathbf{Y}^{1},\mathbf{Y}^{2})$
are two rough paths in rough path space $\Omega_{p}([0,T],\mathbb{R}^{d})$,
and $\mathcal{P}$ is any partition of interval $[0,T]$.

Now we give the continuity of estimator $\widehat{\Gamma}_{T}(\cdot)$
under $p$-variation distance $d_{p}$.

\begin{thm} Let $X$ be a fOU process driven by fBM with the Hurst parameter
$H\in(\frac{1}{3},\frac{1}{2}]$ and let $(X,\mathbb{X})$ be the It\^o
rough path enhancement. Then rough path estimator $\widehat{\Gamma}_{T}:(X(\omega),\mathbb{X}(\omega))\to\widehat{\Gamma}_{T}((X(\omega),\mathbb{X}(\omega)))$
is continuous with respect to $p$-variation distance $d_{p}$ for
$\frac{1}{H}<p<3$. \end{thm} \begin{proof} The statement is a corollary
of Theorem 5.3.1 of Lyons and Qian \cite{LQ02}. \end{proof}

\section{Rough path estimator based on high-frequency data}
\label{sec-DRPE}

In the preceding sections, the estimator we have examined assumes access to continuous observations. However, in practical scenarios, the process is typically observed only at discrete intervals. Consequently, it is imperative to devise an estimator that operates on discrete data. We refer to the work \cite{WXY23}, in which the authors estimated the parameters of the fOU process using discrete-sampled observations in the one-dimensional case. In this section, based on our continuous
rough path estimator, we construct a discrete rough path estimator
and it still has favorable properties. We assume that the fOU process
$X$ can be enhanced to an It\^o rough path $\mathbf{X}=(1,X,\mathbb{X})$
as in Section \ref{sec-RP} and can be observed at discrete time $\{t_{\ell}=\ell h,\ell=0,1,2,\cdots,n\}$,
or equivalently we can get the discrete data $\{(X_{t_{0},t_{1}},\mathbb{X}_{t_{0},t_{1}}),(X_{t_{1},t_{2}},\mathbb{X}_{t_{1},t_{2}}),\cdots,(X_{t_{n-1},t_{n}},\mathbb{X}_{t_{n-1},t_{n}})\}$
in the It\^o sense as in Section \ref{sec-RP}. Here, $n$ is the sample size,
$h=h_{n}$ is the observation frequency, and $t:=nh$ is the time
horizon. We further assume that, as the sample size $n$ tends to infinity,
the observation frequency $h=h_{n}\to0$ and time horizon $t=nh\to\infty$.
In other words, the data is high-frequency. We should
give more assumptions to balance the rate of sample size $n$ and
the frequency $h$ in order to get the good estimator below. Now we give
the theorem of almost sure convergence for our high-frequency rough
path estimator.

\begin{thm}\label{thm-discrete-as} Suppose that the fOU process $X$, 
which is the solution to stochastic differential equation (\ref{fOU-h-2})
with $H\in(\frac{1}{3},\frac{1}{2}]$, can be observed at discrete
time $\{t_{\ell}=\ell h,\ell=0,1,2,\cdots,n\}$, and as the sample size
$n\to\infty$, $n$ and $h$ satisfy that 
\begin{equation}
nh\to\infty,\ h=h_{n}\to0,\ nh^{p}\to0,
\end{equation}
for some $p\in(1,\frac{1+H+\beta}{1+\beta})$, and $0<\beta<1$. Let
\begin{equation}\label{discrete-Gamma}
\widetilde{\Gamma}_{n}^{T}=-\left(\sum_{\ell=0}^{n}(X_{\ell h}\otimes X_{\ell h})h\right)^{-1}\left(\sum_{\ell=0}^{n-1}X_{\ell h}X_{\ell h,(\ell+1)h}+\mathbb{X}_{\ell h,(\ell+1)h}\right),
\end{equation}
where $\widetilde{\Gamma}^{T}$ denotes transpose of matrix $\widetilde{\Gamma}$.
Then
\begin{equation}
\widetilde{\Gamma}_{n}\to\Gamma,\ a.s.,
\end{equation}
as $n\to\infty$. \end{thm}

\begin{proof} Let
\[
\mathcal{L}_{nh}=\int_{0}^{nh}X_{u}\otimes X_{u}du,\ A_{nh}=\int_{0}^{nh}X_{u}\otimes d_{\mathfrak{R}_{1}}\mathbf{X},
\]
and
\[
\widetilde{\mathcal{L}}_{n}=\sum_{\ell=0}^{n}(X_{\ell h}\otimes X_{\ell h})h,\ \widetilde{A}_{n}=\sum_{\ell=0}^{n-1}X_{\ell h}X_{\ell h,(\ell+1)h}+\mathbb{X}_{\ell h,(\ell+1)h}.
\]

By Proposition \ref{limit-x2}, we know that
\begin{equation}
\frac{1}{nh}\mathcal{L}_{nh}=\frac{1}{nh}\int_{0}^{nh}X_{u}\otimes X_{u}du\to C_{1}(H),\ a.s.\ \mathrm{as}\ n\to\infty.\label{L_nh_limit}
\end{equation}
From Corollary \ref{limit-Ito-xdx}, we have that 
\begin{equation}
\frac{1}{nh}A_{nh}=\frac{1}{nh}\int_{0}^{nh}X_{u}\otimes d_{\mathfrak{R}_{1}}\mathbf{X}\to C_{2}(H),\ a.s.\ \mathrm{as}\ n\to\infty.\label{A_nh_limit}
\end{equation}
In what follows, we show that
\begin{equation}
\frac{1}{nh}\left(\mathcal{L}_{nh}-\widetilde{\mathcal{L}}_{n}\right)\to0,\ a.s.,\label{L_n_limit}
\end{equation}
and
\begin{equation}
\frac{1}{nh}\left(A_{nh}-\widetilde{A}_{n}\right)\to0,\ a.s..\label{A_n_limit}
\end{equation}
If so, combining (\ref{L_nh_limit}), (\ref{A_nh_limit}), (\ref{L_n_limit})
and (\ref{A_n_limit}), we can conclude this theorem, that is,
\[
\begin{split}-\widetilde{\mathcal{L}}_{n}^{-1}\widetilde{A}_{n} & =-\left(\frac{1}{nh}\left(\widetilde{\mathcal{L}}_{n}-\mathcal{L}_{nh}\right)+\frac{1}{nh}\mathcal{L}_{nh}\right)^{-1}\\
 & \ \ \ \ \times\left(\frac{1}{nh}\left(\widetilde{A}_{n}-A_{nh}\right)+\frac{1}{nh}A_{nh}\right)\\
 & \to-C_{1}(H)^{-1}C_{2}(H)=\Gamma,\ a.s..
\end{split}
\]
Now showing the limit (\ref{A_n_limit}), we have that 
\[
\begin{split}\frac{1}{nh}\left|A_{nh}-\widetilde{A}_{n}\right| & =\frac{1}{nh}\left|\int_{0}^{nh}X_{u}\otimes d_{\mathfrak{R}_{1}}\mathbf{X}-\sum_{\ell=0}^{n-1}(X_{\ell h}X_{\ell h,(\ell+1)h}+\mathbb{X}_{\ell h,(\ell+1)h})\right|\\
 & =\frac{1}{nh}\left|\sum_{\ell=0}^{n-1}\left(\int_{\ell h}^{(\ell+1)h}X_{u}\otimes d_{\mathfrak{R}_{1}}\mathbf{X}-(X_{\ell h}X_{\ell h,(\ell+1)h}+\mathbb{X}_{\ell h,(\ell+1)h})\right)\right|\\
 & \leq\frac{1}{nh}\sum_{\ell=0}^{n-1}\left|\int_{\ell h}^{(\ell+1)h}X_{u}\otimes d_{\mathfrak{R}_{1}}\mathbf{X}-(X_{\ell h}X_{\ell h,(\ell+1)h}+\mathbb{X}_{\ell h,(\ell+1)h})\right|.
\end{split}
\]
Since
\[
\mathbb{X}_{\ell h,(\ell+1)h}=\int_{\ell h}^{(\ell+1)h}X_{\ell h,u}\otimes d_{\mathfrak{R}_{1}}\mathbf{X}=\int_{\ell h}^{(\ell+1)h}X_{u}\otimes d_{\mathfrak{R}_{1}}\mathbf{X}-X_{\ell h}X_{\ell h,(\ell+1)h},
\]
we have that 
\[
\frac{1}{nh}\left|A_{nh}-\widetilde{A}_{n}\right|=0.
\]
For the limit (\ref{L_n_limit}),
\[
\begin{split}\frac{1}{nh}\left|\mathcal{L}_{nh}-\widetilde{\mathcal{L}}_{n}\right| & =\frac{1}{nh}\left|\int_{0}^{nh}X_{u}\otimes X_{u}du-\sum_{\ell=0}^{n}(X_{\ell h}\otimes X_{\ell h})h\right|\\
 & =\frac{1}{nh}\left|\sum_{\ell=0}^{n-1}\left(\int_{\ell h}^{(\ell+1)h}X_{u}\otimes X_{u}du-(X_{\ell h}\otimes X_{\ell h})h\right)\right|\\
 & \leq\frac{1}{nh}\sum_{\ell=0}^{n-1}\left|\int_{\ell h}^{(\ell+1)h}X_{u}\otimes X_{u}du-(X_{\ell h}\otimes X_{\ell h})h\right|.
\end{split}
\]
Let $F(X_{t})=X_{t}\otimes X_{t}$, and any $0\leq s<t\leq T$. Then, 
by Proposition \ref{Holder-F(X)-lem}, we get that
\[
\left|\int_{s}^{t}F(X_{u})du-F(X_{s})(t-s)\right|\leq CR_{T}T^{\beta}|t-s|^{1+\alpha},\ \forall\alpha\in(0,H).
\]
Taking that $s=\ell h$, $t=(\ell+1)h$, and $T=nh$, we have that 
\[
\left|\mathcal{L}_{nh}-\widetilde{\mathcal{L}}_{n}\right|\leq\sum_{\ell=0}^{n-1}CR_{nh}(nh)^{\beta}h^{1+\alpha}=CR_{nh}n^{1+\beta}h^{1+\alpha+\beta}=CR_{nh}\left(nh^{\frac{1+\alpha+\beta}{1+\beta}}\right)^{1+\beta}.
\]
By assumption, there exists a number $p\in(1,\frac{1+H+\beta}{1+\beta})$
such that $nh^{p}\to0$ and $nh\to\infty$ as $n\to\infty$. Also, $R_{nh}\to0,\ a.s.$.
Thus, we get that $\mathcal{L}_{nh}-\widetilde{\mathcal{L}}_{n}\to0,\ a.s..$
(we may assume that the components of fOU process $X$ are independent,
so we should make an orthogonal transformation for $X$.) Thus, we have
completed the proof of Theorem \ref{thm-discrete-as}. \end{proof}

\begin{rem}
The explicit dependence of the L\'evy area of fOU processes on the drift parameter $\Gamma$ can make it challenging to apply the estimator in practical observations. However, as suggested by equation (\ref{discrete-Gamma}), we can treat it as an equation for $\Gamma$ and solve it iteratively.
\end{rem}

\section*{Acknowledgements}
The authors thank the anonymous reviewers for their valuable comments and suggestions which helped improve and clarify this manuscript.

\section*{Conflict of Interest}
The authors declare no conflict of interest.







\end{document}